\RequirePackage{fix-cm}
\documentclass[twocolumn]{svjour3}          
\smartqed  
\usepackage{amssymb,amsbsy,amsmath,amsfonts,amssymb,amscd}
\usepackage{xcolor}
\usepackage{algorithm,algpseudocode}
\usepackage{enumitem}
\usepackage{xr}

\newcommand{\EE}{\mathbb{E}}
\newcommand{\pos}{\mathrm{pos}}
\newcommand{\Cov}{\mathrm{Cov}}
\newcommand{\MCF}{\mathcal{F}}

\newcommand{\MCG}{\mathcal{G}}
\newcommand{\tp}{\textbf{p}}
\newcommand{\tq}{\textbf{q}}
\newcommand{\tu}{\textbf{u}}
\newcommand{\te}{\textbf{e}}
\newcommand{\tr}{\textbf{r}}
\newcommand{\tv}{\textbf{v}}
\newcommand{\tx}{\textbf{x}}
\newcommand{\Rm}{\mathrm{m}}

\begin{document}

\title{Ensemble Kalman Inversion: mean-field limit and convergence analysis}

\author{Zhiyan Ding\and Qin Li}

\institute{Zhiyan Ding \at
              Mathematics Department, University of Wisconsin-Madison, 480 Lincoln Dr., Madison, WI 53705 USA. \\
              \email{zding49@math.wisc.edu}           
           \and
           Qin Li\at
              Mathematics Department and Wisconsin Institutes of Discoveries, University of Wisconsin-Madison, 480 Lincoln Dr., Madison, WI 53705 USA.\\
              \email{qinli@math.wisc.edu}
}

\date{}

\maketitle

\begin{abstract}
Ensemble Kalman Inversion (EKI) has been a very popular algorithm used in Bayesian inverse problems~\cite{Iglesias_2013}. It samples particles from a prior distribution, and introduces a motion to move the particles around in pseudo-time. As the pseudo-time goes to infinity, the method finds the minimizer of the objective function, and when the pseudo-time stops at $1$, the ensemble distribution of the particles resembles, in some sense, the posterior distribution in the linear setting. The ideas trace back further to Ensemble Kalman Filter and the associated analysis~\cite{Evensen_enkf,Reich2011}, but to today, when viewed as a sampling method, why EKI works, and in what sense with what rate the method converges is still largely unknown.

In this paper, we analyze the continuous version of EKI, a coupled SDE system, and prove the mean field limit of this SDE system. In particular, we will show that 1. as the number of particles goes to infinity, the empirical measure of particles following SDE converges to the solution to a Fokker-Planck equation in Wasserstein 2-distance with an optimal rate, for both linear and weakly nonlinear case; 2. the solution to the Fokker-Planck equation reconstructs the target distribution in finite time in the linear case, as suggested in~\cite{Iglesias_2013}.
\end{abstract}

\maketitle
\section{Introduction}
How to sample from a target distribution is a central challenge in Bayesian inverse problems, especially when the to-be-reconstructed parameter lives on a high dimensional space. Suppose a $1000$-dimensional parameter needs to be reconstructed, and we have a budget of making $10,000$ samples, then how do we design algorithms so that these $10,000$ samples look like they are i.i.d. samples from the posterior distribution?

There are abundant studies in this direction. Traditional methods such as Markov chain Monte Carlo (MCMC) like Metropolis Hastings type algorithm, and sequential Monte Carlo (SMC) have garnered a large amount of investigations both on the theoretical and numerical sides \cite{Doucet2001,Mont,SMont}. Newer methods such as stein variational gradient descent (SVGD) based on Kernelized Stein Discrepancy~\cite{NIPS2016_6338}, the ensemble Kalman inversion (EKI), the ensemble Kalman sampling method (EKS)~\cite{Stuart_gradient,ding2019meanfield_EKS} quickly drew attention from many related areas. There are advantages and disadvantages associated with each method.

In this paper, we study Ensemble Kalman Inversion (EKI) method in depth~\cite{Iglesias_2013}. The method can be viewed as one step in the popular Ensemble Kalman filter (EnKF) method. EnKF was introduced initially for dynamical systems in \cite{DAEnKF,Evensen_enkf,Ghil1981,Evensen2003,firstEnkf}: one sequentially mixes in newly available data and evolve the probability distribution of the to-be-reconstructed parameters along the evolution of the dynamical system \cite{LeGland,law_tembine_tempone}. In each step of EnKF, the method consists of a forecast stage, which amounts to evolving underlying dynamical systems, and the analysis stage, which amounts to adjusting the distribution of states. EKI only studies static problems: one is given a fixed set of data to reconstruct a fixed set of unknown parameters, and thus is comparable to the analysis stage of EnKF. Such connection was first documented in the beautiful paper of \cite{Reich2011} (and the references therein, e.g.~\cite{Bergemann_Reich10_local,Bergemann_Reich10_mollifier}, and was discussed in depth in~\cite{Iglesias_2013} where the authors fully developed the idea into an algorithm. The procedure is rather easy to understand: one i.i.d. samples a fixed number of particles according to the prior distribution and labels them the initial data at $t=0$. The particles are then pushed around according to certain dynamics in (pseudo-)time, hoping at $t=1$ the particles look like they are i.i.d. sampled from the posterior distribution.

The algorithm was designed on the discrete level, with $J$ particles moved around using stepsize $h$, and the number of time steps ($N$ in our paper) is naturally $N=1/h$ to ensure the pseudo-time stops at $1$. The continuous version of the algorithm (with $h\to 0$) represents $J$-coupled SDE systems, for which there are already a number of theoretical studies~\cite{SS,SS2,DCPS}. However, to the authors' understanding, despite some heuristic arguments~\cite{SS,SS2}, there has been no result discussing the $J\to\infty$ limit of the coupled SDE system, and in particular for practical reasons, how this limit connects with the target distribution.

In this paper we will give two results concerning this convergence.
\begin{itemize}
\item We will prove, both in the linear and weak-nonlinear case, the coupled SDE system converges to a Fokker-Planck equation with an optimal rate in Wasserstein 2-metric. The relevant results are Theorem 1 and 2, and the optimality is discussed after the statement of Theorem 1.
\item We will prove that the Fokker-Planck equation connects the prior distribution with the target posterior distribution only in the linear case. This is presented in Corollary 1. The nonlinear case can be vastly more complicated, as discussed in Section 4.2, also see~\cite{law_tembine_tempone}.
\end{itemize}

On the technical level, the first result amounts to showing the mean-field limit of the SDE system. Indeed, we largely rely on the classical Dobrushin's argument, which consists of constructing a ``bridging SDE" and compare the distance between the PDE with the bridging SDE, and the distance between the two SDE systems. The former is an established result in~\cite{Fournier2015}, and the latter amounts to bounding the flux and Brownian motion coefficients, and then looping it back for the Gr\"onwall inequality. The argument, despite being very popular in the mean-field community~\cite{Carrilo2011,Blob,Bolley_Carrillo,Sznitman} to deal with particle systems in chemistry and biology, has rarely been applied to investigate sampling methods. The only exception known to us is~\cite{LuLuNolen} in which the authors proved the continuous version of SVGD is the weak solution to a transport type equation whose equilibrium state at the infinite time is the target posterior distribution. However, due to the Gr\"onwall nature of the argument, the constant blows up in infinite time, while the convergence to the equilibrium requires infinite time. EKI, however, stops at finite time $t=1$, and thus the constant would be finite. Comparing to other mean-field problems emerging in chemistry/biology (such as Cucker-Smale model), the difficulty here mainly comes from the fact that the flux and diffusion coefficients rely on higher moments of the PDE solution, and thus we do not have properties such as Lipschitz continuity for the Gr\"onwall inequality to directly apply.

The way to overcome these technical difficulties is to employ the bootstrapping argument, namely, we assume the convergence is of certain rate, and a lemma (Lemma \ref{lem:bound_x}) is then derived to show that such rate can be tightened. One continues this tightening process till the maximum rate is achieved (Proposition \ref{thm:vj_uj}). The initial convergence rate can be as low as $0$, meaning one only needs the boundedness. This boundedness is shown in Lemma \ref{prop:highmomentforpac2}, Lemma \ref{prop:highmomentbound}, and Corollary \ref{cor:prebound}. Theorem \ref{thm:mean_field} and \ref{thm:moments} are then direct consequences of Proposition \ref{thm:vj_uj}, combined with Proposition \ref{thm:vj_FP}, which itself is a simple application of the celebrated theorem from~\cite{Fournier2015} (cited as Theorem 3 in this paper).

The second result amounts to direct derivation. The argument was hinted in multiple papers~\cite{Reich2011,Evensen2003,Iglesias_2013}, but we have not found explicit derivation in literature.

We would like to mention that in~\cite{Herty} the authors investigated the convergence of the moments using kinetic tools, a relevant class of methods for investigating the convergence of sampling methods; in~\cite{Tabak_schroedinger}, the authors drew the connection with the Schr\"odinger bridge problem, and in~\cite{reich_2019} the authors discuss the transition kernel's dependence in conjunction with dynamics versus analysis. These papers are not directly related to the results presented in this paper, but shed light to understanding of sampling in depth.

In Section~\ref{sec:EKI}, we give a quick overview of the method, and present the continuous version, the SDE of the algorithm. In Section~\ref{sec:results} we summarize our own result, Theorem~\ref{thm:mean_field} and Theorem~\ref{thm:moments}, and present the mean-field limit. In Section~\ref{sec:linear_nonlinear} we discuss the meaning of the result in the linear and nonlinear setting. Section~\ref{sec:mean_field_1} and~\ref{sec:mean_field_2} are dedicated to proving the main theorems. Some calculations are rather technical and we leave them in appendix.

\section{Ensemble Kalman Inversion setup and statement of our result}\label{sec:EKI}
The Ensemble Kalman Inversion (EKI) was initially proposed to be a gradient-free optimization method~\cite{Iglesias_2013}, but has been widely used to find samples that are approximately drawn i.i.d. from the target posterior distribution if one stops the method in finite time. Getting i.i.d. (or approximately i.i.d.) samples from an arbitrarily given target distribution is a challenging task, and obtaining it in finite time makes it even harder. We briefly review the process of the method.

Suppose $u\in\mathbb{R}^L$ is the to-be-reconstructed vector-parameter, and let $\mathcal{G}:\mathbb{R}^L\rightarrow\mathbb{R}^K$ be the parameter-to-observable map, namely:
\[
y=\mathcal{G}(u)+\eta\,,
\]
where $y\in\mathbb{R}^K$ collects the observed data with $\eta$ denotes the noise in the measurement-taking. The general inverse problem amounts to reconstructing $u$ from $y$. The Bayesian inverse problem amounts to reconstructing the distribution of $u$ given $y$ with assumption on the distribution of $\eta$. In this article we let $\eta\sim \mathcal{N}(0,\Gamma)$ be a Gaussian noise independent of $u$.

Denoting the loss functional $\Phi(\cdot;y):\mathbb{R}^L\rightarrow\mathbb{R}$ by
\[
\Phi(u;y)=\frac{1}{2}\left|y-\mathcal{G}(u)\right|^2_\Gamma\,,
\]
where $|\ \cdot\ |_\Gamma:=|\Gamma^{-\frac{1}{2}}\ \cdot\ |\,$. The Bayes' theorem states that the posterior distribution is the (normalized) product of the prior distribution and the likelihood function:
\begin{equation}\label{GT}
\mu_\pos(u)=\frac{1}{Z}\exp{\left(-\Phi(u;y)\right)}\mu_0(u)\,,
\end{equation}
where
\[
Z:=\int_{\mathbb{R}^L}\exp\left(-\Phi(u;y)\right)\mu_0(u)du\,.
\]
Here $Z$ is the normalization factor, $\exp\left(-\Phi(u;y)\right)$ is the likelihood function and $\mu_0$ is the prior density function that collects people's prior knowledge about the distribution of $u$ (suppose it is absolutely continuous with respect to Lebesgue measure for now). This so-called posterior distribution represents the probability measure of the to-be-reconstructed parameter $u$, blending the prior knowledge and the collected data $y$, taking $\eta$, the measurement error into account. See more details in~\cite{Dashti2017,stuart_2010}.

\subsection{Ensemble Kalman Inversion}
The solution of the Bayesian inverse problem is given by~\eqref{GT}, and in practice, one still needs to generate a number of samples that represent this target distribution. These samples can later on be used to estimate quantities such as moments.

There are a large number of algorithms developed towards this end, including the classical MCMC (Markov chain Monte Carlo) method, Sequential Monte Carlo method, and the newly developed SVGD (Stein variational Gradient Descent), birth-death Langevin, Ensemble Kalman Sampling, among many others \cite{NIPS2016_6338,Stuart_gradient,lu2019accelerating}. It is not our intension to compare these different methods. In this paper, we would like to focus on Ensemble Kalman Inversion and give a sharp estimate to the convergence rate of the method. We emphasize that EKI was developed to be an optimization method, and is widely used as a sampling method. We mainly discuss its performance as a sampling method in this article.

In the setup of EKI, a fixed number of particles are sampled according to the prior distribution first, call them $\{u^j_0\}_{j=1}^J$ (with $0$ in the subscript standing for initial time), and these particles are then propagated according to a certain flow defined by the ensemble mean and covariance in pseudo-time. Hopefully by the pseudo-time achieves $1$, the particles can be seen as i.i.d. drawn from the posterior distribution. The algorithm is summarized in Algorithm~\ref{alg:EKI}.

\begin{algorithm}[htb]
\caption{\textbf{Ensemble Kalman Inversion}}\label{alg:EKI}
\begin{algorithmic}
\State \textbf{Preparation:}
\State 1. Input: $J\gg1$; $h\ll1$ (time step); $N=1/h$ (stopping index); $\Gamma$; and $y$ (data).
\State 2. Initial: $\{u^j_0\}$ sampled from initial distribution induced by density function $\mu_0$.

\State \textbf{Run: } Set time step $n=0$;
\State \textbf{While} $n<N$:

\noindent 1. Define empirical means and covariance:
\begin{equation}\label{eqn:ensemble_alg}
\begin{aligned}
\quad&\overline{u}_n=\frac{1}{J}\sum^J_{j=1}u^j_n\,,\ \text{and}\ \overline{\MCG}_n=\frac{1}{J}\sum^J_{j=1}\MCG(u^j_n)\,,\\
\quad&C^{pp}_n(u)=\frac{1}{J}\sum^J_{j=1}\left(\MCG(u^j_n)-\overline{\MCG}_n\right)\otimes \left(\MCG(u^j_n)-\overline{\MCG}_n\right)\,,\\
\quad&C^{up}_n(u)=\frac{1}{J}\sum^J_{j=1}\left(u^j_n-\overline{u}_n\right)\otimes \left(\MCG(u^j_n)-\overline{\MCG}_n\right)\,.
\end{aligned}
\end{equation}

\noindent2. Artificially perturb data (with $\xi^j_{n+1}$ drawn $i.i.d.$ from $\mathcal{N}(0,h^{-1}\Gamma)$):
\[
\quad y^j_{n+1}=y+\xi^j_{n+1},\quad\forall 1\leq j\leq J\,.
\]

\noindent3. Update (set $n\to n+1$)
\begin{equation}\label{eqn:update_ujn}
\begin{aligned}
&\quad r^j_n=y^{j}_{n+1}-\MCG(u^j_n),\\
&\quad u^j_{n+1}=u^j_n+C^{up}_n(u_n)\left(C^{pp}_n(u_n)+h^{-1}\Gamma\right)^{-1}r^j_n\,,
\end{aligned}
\end{equation}
for all $1\leq j\leq J$.
\State \textbf{end}
\State \textbf{Output:} $\{u^j_N\}$.
\end{algorithmic}
\end{algorithm}

Prior to running the algorithm, one first specifies the number of samples needed (denote by $J$), and the number of steps one can take (denote by $N$). The time-step size, then is simply $h=1/N$. This is to ensure $t=1$ is the final time. So in total, there are two parameters in the algorithm:
\begin{itemize}
\item[1:] The pseudo-time-step $h$.
\item[2.] The number of particles $J$.
\end{itemize}

Along the evolution, at each time step, one computes the sample mean and covariance in~\eqref{eqn:ensemble_alg}, and uses them to move the samples around according to~\eqref{eqn:update_ujn}.

Upon finishing the algorithm in $N$ steps, one obtains a list of particles $\{u^j_N\}_{j=1}^J$ and defines the ensemble distribution:
\begin{equation}\label{eqn:ensemble}
M_{u}=\frac{1}{J}\sum^J_{j=1}\delta_{u^j_N}\,.
\end{equation}
It is our goal, in this article to show in both linear and nonlinear setup, when and how $M_u$ approximates target posterior distribution induced by posterior density function $\mu_\pos$. 

There are two parameters in the algorithm, and thus the convergence result of the algorithm to the posterior distribution should be established in the $h\to0$ and $J\to\infty$ limit. The $h\to0$ limit was discussed in~\cite{SS}, also see our Section 2.2, and in this paper we study the $J\to\infty$ limit.

\begin{remark}
Four comments are in order:
\begin{itemize}
\item[1.] We emphasize that $N$ and $h$ satisfy a certain relation: $Nh = 1$, and thus $N$ is not a free parameter. This fact is easily overlooked. In fact, in all the previous theoretical studies that we found~\cite{SS,DCPS}, people have been looking for convergence result where $h\to0$ first and $N\to\infty$ afterwards. Namely it is
\[
\lim_{N\to\infty}\lim_{h\to 0}\quad\text{instead of}\quad\lim_{Nh=1,h\to0}
\]
that has been studied. These works lay the theoretical foundation for ours, and builds wellposedness theory for the underlying SDE, but we would like to emphasize, however, that the two limits do not commute. Exactly for this reason, when one considers $\lim_{Nh=1,h\to0}$, a posterior distribution is obtained, but when the two limits are taken separately, the ``collapsing" phenomenon is observed~\cite{SS,Iglesias_2013}. In this article, we stick to the finite time $t=Nh=1$ regime.
\item[2.] We do not aim at comparing different methods, but one immediate advantage of this method over MCMC or other classical sampling method is worth of mentioning: in this method, the number of samples are fixed, and the number of steps are also fixed. So instead of tracing the error in time and terminating the process on-the-fly whenever tolerance is met, the number of particles is pre-set, and thus the numerical cost is known ahead of the computation. Indeed, exactly because of this, the error analysis is rather crucial: based on the error analysis, one can pre-determine the proper values of $J$ and $h$.
\item[3.] EKI shares some similarity with a very famous data assimilation method called Ensemble Kalman Filter~\cite{Evensen_enkf}, which was itself derived from Kalman filter with the mean and the covariance replaced by their ensemble versions. One main difference between EKI and EnKF is that EKI looks for solution to a static problem, and the dynamics is built in pseudo-time. EnKF, however, tries to blend information from the underlying dynamics, characterized by ODE/PDE/SDE, and the collected data, using the Bayesian formulation. The time in EnKF is real. A beautiful set of analysis can be found in~\cite{LeGland,law_tembine_tempone,Ernst}. These works provide theoretical studies in the ensemble Kalman framework. However, these results consider discrete case where the time stepsize $h=1$. On the contrary, we study the continuum limit with $h\rightarrow0$, and a lot of technicalities are associated with SDE's mean-field limit analysis, making the previous results not particularly useful in our setting.


\item[4.] Similar to the EnKF, EKI also tries to translate particles from one distribution to another, and records only the first two moments (mean and covariance). If the distribution fails to be a Gaussian along the evolution, information carried by the higher moments is simply removed from the system, leading to numerical error unavoidably. If the nonlinearity is weak, higher moments could be potentially bounded and there is still hope to control the EKI's mean-field limit. We will explain this in better detail in Section 3, when we present the weakly nonlinear assumption in~\eqref{linear}.
\end{itemize}
\end{remark}

\subsection{Continuum limit and dynamical system of $\{u_t^j\}$}
EKI is an algorithm with discrete-in-time updates. Formally let the time step $h\to 0$, equation~\eqref{eqn:update_ujn} becomes:
\begin{equation}\label{eqn:SDE_general}
du^j_t=C^{up}(u_t)\Gamma^{-1}\left(y-\MCG(u^j_t)\right)dt+C^{up}(u_t)\Gamma^{-\frac{1}{2}}dW^j_t\,,
\end{equation}
where
\[
C^{up}(u)=\frac{1}{J}\sum^J_{j=1}\left(u^j-\overline{u}\right)\otimes \left(\MCG(u^j)-\overline{\MCG}\right)\,
\]
with
\[
\overline{u}=\frac{1}{J}\sum^J_{j=1}u^j,\quad \overline{\MCG}=
\frac{1}{J}\sum^J_{j=1}\MCG(u^j)\,.
\]
Here $\otimes$ means the first argument is viewed as a column vector while the second is viewed as the row vector. 

Indeed, as shown in~\cite{SS,Blomker}, the method~\eqref{eqn:update_ujn} can be viewed as the Euler-Maruyama discretization of the SDE.

Let $\Omega$ be the sample space and $\MCF_0$ being the $\sigma$-algebra: $\sigma\left(u^j(t=0),1\leq j\leq J\right)$, then the filtration is introduced by the dynamics:
\begin{equation*}
\MCF_t=\sigma\left(u^j(t=0),W^j_s,1\leq j\leq J,s\leq t\right)\,.
\end{equation*}

In~\cite{DCPS}, the authors showed the wellposedness of the SDE system under the linear assumption ($\mathcal{G}=Au$). The techniques, when combined with boundedness of moments, should work even when $\mathcal{G}$ is nonlinear. In the later section (in particular, Lemma 2), we will prove the boundedness of the moments. However, how to explicitly incorporate these with the techniques in~\cite{DCPS} for the wellposedness is beyond the focus of the current paper. In~\cite{SS,Blomker}, the authors formally derive the continuum limit of the method and arrived at the SDE. The proof has not been made rigorous. Indeed for the convergence of the Euler-Maruyama discretization, strong assumptions are imposed on the coefficients (transport and Brownian motion), and the nonlinearity induced in the covariance matrix makes the proof highly nontrivial. We believe under certain condition on the target distribution, this could be made possible, but it is also not directly related to deriving and proving the mean-field limit, and will be omitted from the current paper. A similar result under the EnKF framework~\cite{lange2019continuous} could potentially be useful in this direction.

In this paper, we start with the SDE, and we will analyze its mean-field limit as $J\to\infty$ in the Wasserstein-2 metric. The limit is characterized by a Fokker-Planck (FP) type equation, and we will show, in the linear setting, such FP equation recovers the posterior distribution and in the nonlinear setting, it deviates from the posterior distribution by a weight factor. 

\section{Main theorem and mean field limit}\label{sec:results}
We present our main theorem in this section.

To do so we first unify the notations. In the paper we denote $\mathbb{E}$ the expectation in the probability space $\left(\Omega,\MCF_t,\mathbb{P}\right)$ and often use $\rho_t$ as a short notation for $\rho(t,u)$. For any vectors $\{m^j\}^J_{j=1}$ and $\{n^j\}^J_{j=1}$, we denote
\[
\overline{m}=\frac{1}{J}\sum^J_{j=1}m^j
\]
and
\[
\Cov_{m,n}=\frac{1}{J}\sum^J_{j=1}\left(m^{j}-\overline{m}\right)\otimes\left(n^{j}-\overline{n}\right)\,,
\]
and denote $\mathrm{Cov}_m=\Cov_{m,m}$. Here $\otimes$ means the first argument is viewed as a column vector while the second is viewed as the row vector. Similarly, for any probability density function $\rho$ and function $g$, we denote
\[
\EE_{\rho}=\int_{\mathbb{R}^L}u\rho(u)du,\quad \EE_{g,\rho}=\int_{\mathbb{R}^L}g(u)\rho(u)du\,,
\]
\[
\Cov_{\rho}=\int_{\mathbb{R}^L}\left(u-\mathbb{E}_\rho\right)\otimes \left(u-\mathbb{E}_\rho\right)\rho(u)du\,,
\]
and
\[
\Cov_{\rho,g}=\int_{\mathbb{R}^L}\left(u-\mathbb{E}_\rho\right)\otimes \left(g(u)-\mathbb{E}_{g,\rho}\right)\rho(u)du\,.
\]
Apparently $\Cov_{g,\rho}=\Cov^\top_{\rho,g}$. 

The distance we use to quantify the ``smallness" is the Wasserstein 2-metric:
\begin{definition} Let $\upsilon_1,\upsilon_2$ be two probability measures in $\left(\mathbb{R}^L,\mathcal{B}_{\mathbb{R}^L}\right)$, then the $W_2$-Wasserstein distance between $\upsilon_1,\upsilon_2$ is defined as
\begin{equation*}
W_2(\upsilon_1,\upsilon_2):=\left(\inf_{\gamma\in\Gamma(\upsilon_1,\upsilon_2)}\int_{\mathbb{R}^L\times\mathbb{R}^L}|x-y|^2d\gamma(x,y)\right)^{\frac{1}{2}},
\end{equation*}
where $\Gamma(\upsilon_1,\upsilon_2)$ denotes the collection of all measures on $\mathbb{R}^L\times\mathbb{R}^L$ with marginals $\upsilon_1$ and $\upsilon_2$ for $x$ and $y$ respectively. Here $\upsilon_i$ can be either general probability measures or the measures induced by probability density functions $\upsilon_i$.
\end{definition}

We also assume weak nonlinearity, meaning there is a matrix $A\in\mathcal{L}(\mathbb{R}^L,\mathbb{R}^K)$ such that

\begin{equation}\label{linear}
\mathcal{G}(u)=Au+\Rm(u)\,,
\end{equation}
where $\Rm(u):\mathbb{R}^L\to\mathbb{R}^K$ is a smooth bounded function satisfying
\begin{equation*}
\text{Range}(\Rm)\perp_{\Gamma^{-1}}\text{Range}(A),\ \left|\Rm(u)\right|+\left|\nabla_u\Rm(u)\right|\leq M\,,
\end{equation*}
with some constant $M>0$ in $\mathbb{R}^L$, and $a\perp_{\Gamma^{-1}}b$ means $a^\top \Gamma^{-1}b=0$ and $a^\top$ is to the take transpose of $a$. This assumption plays a crucial role in the later proofs: it eliminates the cross-terms such as $\Rm^\top\Gamma^{-1}A$ in the posterior distribution, and thus put $\Rm$ entirely in the perpendicular direction of $\text{Range}(A)$. The $\Rm^\top\Gamma^{-1}\Rm$ terms are then controlled using the boundedness condition, boiling the analysis down to the linear situation.

We further denote the ``closest" solution of the linear component to be $u^\dagger$, and $r$ the corresponding noise, then
\begin{equation}\label{eqn:u_dagger}
y=Au^\dagger+r,\ \text{with}\ r^\top\Gamma^{-1}\text{range}(A)=0\,,
\end{equation}
then the loss functional is also explicit:
\[
\begin{aligned}
\Phi\left(u;y\right)=&\,\frac{1}{2}\left(u^\dagger-u\right)^\top A^\top \Gamma^{-1}A(u^\dagger-u)\\
&+\frac{1}{2}\left(r-\Rm(u)\right)^\top\Gamma^{-1}\left(r-\Rm(u)\right)\,,
\end{aligned}
\]
where we used the fact that $\Rm\perp_{\Gamma^{-1}}A$, $r\perp_{\Gamma^{-1}}A$.

Under such weakly nonlinear assumption~\eqref{linear}, the dynamical system of $\{u^j_t\}$, written in~\eqref{eqn:SDE_general} can be expanded:
\begin{equation}\label{eqn:SDE}
\begin{aligned}
{du^{j}_t}=&\,\mathrm{Cov}_{u_t,u_t}A^\top  \Gamma^{-1}A\left(u^\dagger-u^{j}_t\right)dt\\&+\mathrm{Cov}_{u_t,u_t}A^\top  \Gamma^{-\frac{1}{2}}dW^{j}_t\\
&+\mathrm{Cov}_{u_t,\Rm}\Gamma^{-1}\left(r-\Rm(u)\right)dt\\
&+\mathrm{Cov}_{u_t,\Rm}\Gamma^{-\frac{1}{2}}dW^{j}_t\,.
\end{aligned}
\end{equation}

Our main theorem states as the following:
\begin{theorem}[Main result 1: mean-field limit]\label{thm:mean_field}
Under the weakly nonlinear assumption~\eqref{linear}, the mean field limit of $M_{u_t}$ is the probability distribution induced by $\rho(t,u)$. Here $M_{u_t}$ is the ensemble distribution of $\{u^j_t\}$ as defined in~\eqref{eqn:ensemble} and $\rho(t,u)$ is the strong solution to the following Fokker-Planck equation:
\begin{equation}\label{eqn:muPDE}
\left\{\begin{aligned}
\partial_t\rho&=-\nabla_u\cdot\left(\left(y-\mathcal{G}(u)\right)^\top \Gamma^{-1}\mathrm{Cov}_{\mathcal{G},\rho_t}\rho\right)\\
&\quad\,+\frac{1}{2}\mathrm{Tr}\left(\mathrm{Cov}_{\rho_t,\mathcal{G}}\Gamma^{-1}\mathrm{Cov}_{\mathcal{G},\rho_t}\mathcal{H}_u(\rho)\right)\,,\\
\rho(0,u)&=\mu_0(u)
\end{aligned}\right.
\end{equation}
where $\mu_0$ is the prior density function, $\mathcal{H}_u(\rho)$ is Hessian of $\rho$.

More specifically, assume $\mu_0$ is $\mathcal{C}^2$, and for any $p>0$, $\mu_0$ satisfies
\[
\int_{\mathbb{R}^L}|u|^p\mu_0(u)du=C_p<\infty\,.
\]
If $\{u^j_0\}$ are i.i.d. sampled from the measure induced by $\mu_0$, then for any $t<\infty$ and any $\epsilon>0$, there is a constant $C_\epsilon(t)$ independent of $J$ such that: 
\begin{equation*}
\mathbb{E}\left(W_2(M_{u_t},\rho(t))\right)\leq C_\epsilon(t)
\left\{
\begin{aligned}
&J^{-\frac{1}{2}+\epsilon},\ L\leq 4\\
&J^{-2/L},\ L>4
\end{aligned}\,.
\right.
\end{equation*}
\end{theorem}

The significance of the result is apparent. 1. When the number of samples $J$ is big enough, the ensemble distribution of $\{u^j_t\}$, the continuous version of EKI can be viewed approximately the solution to the Fokker-Planck equation~\eqref{eqn:muPDE}. So to analyze the long time large sample properties of EKI is boiled down to analyzing a Fokker-Planck equation~\eqref{eqn:muPDE}. The analysis for the latter is very rich, and the literature encompasses the wellposedness, the existence of the equilibrium and the convergence rate in time. All these could direct us in better understanding the algorithm. 2. We give the specific rate of convergence. For $L\leq 4$ in particular, the convergence rate is essentially $J^{-\frac{1}{2}}$. This is the optimal rate one can hope for from a Monte Carlo sampling method. For the case $L> 4$, we believe the result is also optimal. Indeed, as will shown in Section 4, by setting up a dynamical system $\{v^j_t\}$ that strictly follow the flow of the PDE, one expects the best representation of the PDE on the particle level, but yet, $W_2(M_v,\rho)$ is at best of $J^{-2/L}$, according to~\cite{Fournier2015}. So the theorem above is essentially saying that $\{u^j_t\}$, while being accessible, is not worse than $\{v^j_t\}$, and thus obtains the best possible convergence rate.

We do have to mention, however, the theorem quantifies the Wasserstein distance. It is a very strong measure. In practice, it is sufficient to have a number of particles that can characterize the weak convergence. For this practical purpose, we also show the following theorem:
\begin{theorem}
[Main result 2: weak convergence]\label{thm:moments}
Under the weakly nonlinear assumption~\eqref{linear}, $M_{u_t}$ weakly converge to the probability distribution induces by $\rho(t,u)$ with the optimal rate, namely: given any $l$-Lipschitz function $f$, for any $\epsilon>0$, there is a constant $C_\epsilon(l,f(\vec{0}),t)$ independent of $J$ such that: for any $t<\infty$
\begin{equation}\label{weakconvergence}
\begin{aligned}
&\left(\EE\left|\int f(u)\left[M_{u_t}-\rho(t,u)\right] du\right|^2\right)^{\frac{1}{2}}\\
 \leq &C_\epsilon(l,f(\vec{0}),t)J^{-\frac{1}{2}+\epsilon}
\end{aligned}\,.
\end{equation}
Here $M_{u_t}$ is the ensemble distribution~\eqref{eqn:ensemble} and $\rho$ solves~\eqref{eqn:muPDE}.
\end{theorem}
This result significantly strengthen the convergence rate, and eliminates the dimension $L$-dependence.

\section{Linear and nonlinear setups}\label{sec:linear_nonlinear}
Before proving the two theorems, we present here how to interpret them in linear and nonlinear setups.

\subsection{Linear setup}
This is the setup in which we consider $\Rm = 0$, meaning $\mathcal{G}(u) = Au$, and the initial condition $\mu_0$ is a Gaussian density function. When this happens, on one hand, the entire process of the FP evolution is a Gaussian process, and on the other, the posterior distribution is also a Gaussian, and thus one would expect the complete reconstruction.

Indeed let us follow~\cite{SS} and define:
\begin{equation}\label{actualmu}
\mu(t,u)=\frac{1}{Z(t)}\exp\left(-t\Phi(u;y)\right)\mu_0(u)\,,
\end{equation}
where $Z(t):=\int_{\mathbb{R}^L}\exp\left(-t\Phi(u;y)\right)\mu_0(u)du$ is the normalization factor, then it is clear that
\[
\mu(t=0,u) = \mu_0\,,\ \text{and}\ \mu(t=1,u) = \mu_\pos\,,
\]
meaning this new definition~\eqref{actualmu} finds a smooth transition that moves the prior distribution to the posterior, and exactly reconstructs our target distribution at precisely $t=1$. With more derivation, one can actually show this is a strong solution to the Fokker-Planck equation, meaning $\rho(t,u)=\mu(t,u)$ satisfies~\eqref{eqn:muPDE}, and $\rho(t=1,u)$ is the posterior density function under the linear assumption.

This quickly leads to a corollary of the main theorem:
\begin{corollary}\label{thm:mainresult}
Under assumption~\eqref{linear} with $\Rm(u)=\vec{0}$, and $\{u^j_0\}$ are i.i.d. sampled from a Gaussian distribution induced by density function $\mu_0(u)$, then for any $\epsilon>0$, there exists $J(\epsilon)>0$, such that for any $J>J(\epsilon)$
\begin{equation*}
\mathbb{E}(W_2(\mu_\pos(u),M_{u_1})\leq \epsilon\,,
\end{equation*}
where $M_{u_1}$, defined in~\eqref{eqn:ensemble_alg}, is the ensemble distribution of $\{u^j_{1}\}$, the SDE~\eqref{eqn:SDE} solution, and $\mu_\pos$ is the posterior density function induces the posterior distribution.
\end{corollary}

The corollary is direct consequence of Theorem~\ref{thm:mean_field} and we omit the proof. To show that $\mu(t,u)$ is the solution to the PDE~\eqref{eqn:muPDE} amounts to calculating its time and first two derivatives in $u$ and plugging them in~\eqref{eqn:muPDE} to balance the terms out. For the completeness of the paper, we present the derivation briefly below. Without loss of generality, we assume $y=Au^\dagger$ with $r=0$.

Taking the time derivative, we have:
\begin{equation}\label{mut}
\partial_t\mu(t,u)=-\Phi\left(u;y\right)\mu(t,u)-\frac{\partial_t Z(t)}{Z(t)}\mu(t,u)\,,
\end{equation}
where, under the linearity assumption:
\begin{equation*}
\Phi\left(u;y\right)=\left(u^\dagger-u\right)^\top A^\top \Gamma^{-1}A(u^\dagger-u)/2\,,
\end{equation*}
and
\begin{equation*}
\begin{aligned}
\frac{\partial_tZ}{Z}=&\int -(u-\mathbb{E}_{\mu_t})^\top A^\top \Gamma^{-1}A(u-\mathbb{E}_{\mu_t})^\top /2\mu du\\
&+\int -(\mathbb{E}_{\mu_t}-u^\dagger)^\top A^\top \Gamma^{-1}A(\mathbb{E}_{\mu_t}-u^\dagger)^\top /2\mu du\\
=&-\text{Tr}\left[\mathrm{Cov}_{\mu_t} A^\top \Gamma^{-1}A\right]/2\\
&-\left(u^\dagger-\mathbb{E}_{\mu_t}\right)^\top A^\top \Gamma^{-1}A\left(u^\dagger-\mathbb{E}_{\mu_t}\right)/2\,.
\end{aligned}
\end{equation*}
Similarly the gradients in $u$ are:
\begin{equation*}
\begin{aligned}
\nabla_u\mu(t,u)=&tA^\top \Gamma^{-1}A(u^\dagger-u)\mu(t,u)\\
&+\Gamma^{-1}_0\left(u_0-u\right)\mu(t,u)
\end{aligned}\,,
\end{equation*}
and the hessian is:
\begin{equation*}
\begin{aligned}
\mathcal{H}_u\mu=(\mathrm{Cov}_{\mu_t})^{-1}\left(-I+(u-\mathbb{E}_{\mu_t})(u-\mathbb{E}_{\mu_t})^\top(\mathrm{Cov}_{\mu_t})^{-1}\right)\mu
\end{aligned}\,.
\end{equation*}

Putting them back into~\eqref{eqn:muPDE}, one has
\begin{equation*}
\begin{aligned}
\partial_t\mu&+\nabla_u\cdot\left(\left(u^\dagger-u\right)^\top A^\top \Gamma^{-1}A\mathrm{Cov}_{\mu_t}\mu\right)\\
&-\frac{1}{2}\mathrm{Tr}\left(\mathrm{Cov}_{\mu_t} A^\top\Gamma^{-1}A\mathrm{Cov}_{\mu_t}\mathcal{H}_u(\mu)\right)\\
=\partial_t\mu&+\left(u^\dagger-u\right)^\top A^\top \Gamma^{-1}A\mathrm{Cov}_{\mu_t}\nabla_u\mu\\
&+\nabla_u\cdot\left(\left(u^\dagger-u\right)^\top A^\top\Gamma^{-1}A\mathrm{Cov}_{\mu_t}\right)\mu\\
&-\frac{1}{2}\mathrm{Tr}\left(\mathrm{Cov}_{\mu_t} A^\top\Gamma^{-1}A\mathrm{Cov}_{\mu_t}\mathcal{H}_u(\mu)\right)\\
=\text{term I}&+\text{term II}+\text{term III}+\text{term IV}\,.
\end{aligned}
\end{equation*}

Term III becomes to $\text{Tr}\left[\mathrm{Cov}_{\mu_t} A^\top \Gamma^{-1}A\right]\mu$, and Term IV turns to:
\[
\begin{aligned}
&-\frac{1}{2}\mathrm{Tr}\left(\mathrm{Cov}_{\mu_t} A^\top\Gamma^{-1}A\mathrm{Cov}_{\mu_t}\mathcal{H}_u(\mu)\right)\\
=&\ \frac{1}{2}\mathrm{Tr}\left(\mathrm{Cov}_{\mu_t} A^\top\Gamma^{-1}A\right)\mu+\frac{1}{2}\left|A(u-\mathbb{E}_{\mu_t})\right|^2_{\Gamma}\mu\,.
\end{aligned}
\]
To handle term II, we have:
\begin{equation*}
\begin{aligned}
&\left(u^\dagger-u\right)^\top A^\top\Gamma^{-1}A\mathrm{Cov}_{\mu_t}\nabla_u\mu\\
=&t\left(u^\dagger-u\right)^\top A^\top\Gamma^{-1}A\mathrm{Cov}_{\mu_t} A^\top\Gamma^{-1}A(u^\dagger-u)\mu\\
&+\left(u^\dagger-u\right)^\top A^\top\Gamma^{-1}A\mathrm{Cov}_{\mu_t}\Gamma^{-1}_0\left(u_0-u\right)\mu\\
=&t\left(u^\dagger-u\right)^\top A^\top\Gamma^{-1}A(u^\dagger-u)\mu\\
&-\left(u^\dagger-u\right)^\top A^\top\Gamma^{-1}A\mathrm{Cov}_{\mu_t}\Gamma^{-1}_0\left(u^\dagger-u_0\right)\mu\\
=&t\left(u^\dagger-u\right)^\top A^\top\Gamma^{-1}A(u^\dagger-u)\mu\\
&-\left(u^\dagger-u\right)^\top A^\top\Gamma^{-1}A\left(u^\dagger-\mathbb{E}_{\mu_t}\right)\mu\,.
\end{aligned}
\end{equation*}

Adding all the terms up, we find the summation being $0$, making $\mu$ the strong solution to the PDE~\eqref{eqn:muPDE}.

\subsection{Nonlinear setup}
In the weakly nonlinear situation, Theorem~\ref{thm:mean_field} still holds true, however, $\mu(t,u)$, as defined in~\eqref{actualmu}, despite smoothly connects the prior and the target distribution, is no longer the solution to the PDE. Indeed, if we plug it in, define the operator
\[
\begin{aligned}
\mathcal{L}\left[\mu\right]=&\partial_t\mu(t,u)+\nabla_u\cdot\left(\left(y-\mathcal{G}(u)\right)^\top\Gamma^{-1}\mathrm{Cov}_{\mathcal{G},\mu_t}\mu\right)\\
&-\frac{1}{2}\mathrm{Tr}\left(\mathrm{Cov}_{\mu_t,\mathcal{G}}\Gamma^{-1}\mathrm{Cov}_{\mathcal{G},\mu_t}\mathcal{H}_u(\mu)\right)
\end{aligned}\,,
\]
we have $\mathcal{L}\mu \neq 0$ as it is in the linear case, but rather
\begin{equation*}\label{eqn:muPDE_nonlinear}
\mathcal{L}\left[\mu\right]=\left[\mathcal{R}_1(t,u)+\mathcal{R}_2(t,u)+\mathcal{R}_3(t,u)\right]\mu(t,u)\,.
\end{equation*}
The remaining term are:
\[
\begin{aligned}
\mathcal{R}_1(t,u)=&\frac{1}{2}\mathrm{Tr}\left\{\Cov_{\mathcal{G},\mathcal{G}}\Gamma^{-1}\right\}-\mathrm{Tr}\left\{\nabla \MCG(u)\Gamma^{-1}\mathrm{Cov}_{\mathcal{G},\mu_t}\right\}\\
&+\frac{1}{2}\mathrm{Tr}\left\{\mathrm{Cov}_{\mu_t,\mathcal{G}}\Gamma^{-1}\mathrm{Cov}_{\mathcal{G},\mu_t}\mathcal{V}(u)\right\}
\end{aligned}\,,
\]
\[
\begin{aligned}
\mathcal{R}_2(t,u)=&\ \frac{1}{2}\left(y-\overline{\MCG}\right)^\top\Gamma^{-1}\left(y-\overline{\MCG}\right)\\
&-\frac{1}{2}\left(y-\MCG(u)\right)^\top\Gamma^{-1}\left(y-\MCG(u)\right)\\
&+\left(y-\MCG(u)\right)\Gamma^{-1}\mathrm{Cov}_{\mathcal{G},\mu_t}\mathcal{V}(u)\\
&-\frac{1}{2}\mathcal{V}^\top(u)\mathrm{Cov}_{\mu_t,\mathcal{G}}\Gamma^{-1}\mathrm{Cov}_{\mathcal{G},\mu_t}\mathcal{V}(u)\,,
\end{aligned}
\]
\[
\mathcal{R}_3(t,u)=-\frac{t}{2}\mathrm{Tr}\left\{\mathrm{Cov}_{\mu_t,\mathcal{G}}\Gamma^{-1}\mathrm{Cov}_{\mathcal{G},\mu_t}\mathcal{W}(u)\right\}\,
\]
with
\[
\mathcal{V}(u)=t\left(\nabla \MCG(u)\right)^\top\Gamma^{-1}\left(y-\MCG(u)\right)-\Gamma^{-1}_0\left(u-u_0\right),
\]
\[
\mathcal{W}(u)\in\mathbb{R}^{L\times L},\ \text{with}\quad (\mathcal{W}(u))_{:,i}=\partial_i\nabla\MCG\Gamma^{-1}(y-\MCG(u))\,.
\]

This equation defers from the PDE by the three weight terms $\mathcal{R}_i$. In some sense, this is a negative result. It suggests that density of the mean field limit of $M_{u_t}$, proved to be $\rho(t,u)$, defers from $\mu(t,u)$ by the weight terms $\mathcal{R}_i$, that could potentially bring an $O(1)$ effects. The question then comes down to bounding the effects of $\mathcal{R}_i$ and showing them to be small in certain scenarios. This is, however, not within the realm of deriving and proving the mean-field limit, and is beyond the focus of this paper. More discussion can be found in~\cite{Ernst,ding2020ensemblecorrect,law_tembine_tempone}.

\section{Proof of Theorem~\ref{thm:mean_field}, Part I}\label{sec:mean_field_1}
We now start proving the theorem. For notation-wise simplicity, we consider $0\leq t\leq 1$, and all proofs can be easily extended to $1<t<\infty$. To a large extent, we rely on a ``bridge" to connect $\rho$, the solution to the PDE~\eqref{eqn:muPDE}, and the $\{u^j_t\}$ system, the solution to the SDE~\eqref{eqn:SDE}. The ``bridge" is another dynamical system, termed $\{v^j_t\}$ that follows the exact the same flow defined by~\eqref{eqn:muPDE}, meaning the coefficient in $\{v^j_t\}$ are defined by $\rho(t,u)$ and regarded as given a-priori.

Intuitively since $\{v^j_t\}$ follows the flow of the PDE, it carries the PDE information, and thus its ensemble distribution should be close to the measure induced by $\rho$. This is discussed in Proposition~\ref{thm:vj_FP}. $\{v^j_t\}$ inherits properties of $\rho$, such as boundedness of moments, as will be presented in Lemma~\ref{prop:highmomentforpac2}. Since $\{v^j\}$ and $\{u^j\}$ are both dynamical systems, the comparison is boiled down to the stability analysis for SDE systems, and this part of the result is presented in Proposition~\ref{thm:vj_uj}.

The proof of the theorem is thereby divided into two sections, here and the subsequent one: in this section, we show the closeness of $\{v^j_t\}$ and $\rho_t$, and in the following we show the closeness of $\{v^j_t\}$ and $\{u^j_t\}$. Both results are characterized in $W_2$-metric, and the combination of the two naturally leads to the proof of Theorem~\ref{thm:mean_field}, \ref{thm:moments}.

In this section in particular, we discuss the properties of the Fokker-Planck equation and give some estimates of the moments in Section~\ref{sec:FP_prior}. We then discuss $\{v^j_t\}$ system in Section~\ref{sec:vj_FP}.

\subsection{Properties of the Fokker-Planck equation}\label{sec:FP_prior}
We would like to show the boundedness of moments of $\rho(t,u)$, the solution to \eqref{eqn:muPDE}. We start with the covariance first:
\begin{lemma}\label{postvarianceestimate}
Under weakly nonlinear assumption \eqref{linear}, we have: for $0\leq t\leq 1$
\begin{equation}\label{eqn:covbound}
\|\Cov_{\rho_t}\|_2\leq C,\quad \|\Cov_{\rho_t,\mathcal{G}}\|_2\leq C\,,
\end{equation}
where $C$ is a constant independent of $t$ and $\rho(t,u)$ is the solution to \eqref{eqn:muPDE}.
\end{lemma}
\begin{proof}
First, by the weakly-nonlinear assumption~\eqref{linear}, there is an $M>0$:
\[
\left|\mathcal{G}(u_1)-\mathcal{G}(u_2)\right|\leq \max(\|A\|_2,M)|u_1-u_2|\,.
\]
Multiplying $\|u-\mathbb{E}_{\rho_t}\|^2$ on both sides of \eqref{eqn:muPDE} and take integral, we have
\begin{equation*}
\begin{aligned}
&\partial_t\int_{\mathbb{R}^K}\|u-\mathbb{E}_{\rho_t}\|^2\rho(t,u)du\\
=&\int_{\mathbb{R}^K}2\left(y-\mathcal{G}(u)\right)^\top \Gamma^{-1}\mathrm{Cov}_{\mathcal{G},\rho_t}\left(u-\mathbb{E}_{\rho_t}\right)\rho\\
&+\mathrm{Tr}\left(\mathrm{Cov}_{\rho_t,\mathcal{G}}\Gamma^{-1}\mathrm{Cov}_{\mathcal{G},\rho_t}\right)\rho du\,\\
=&\int_{\mathbb{R}^K}-2\left(\mathcal{G}(u)-\mathbb{E}_{\mathcal{G},\rho_t}\right)^\top \Gamma^{-1}\mathrm{Cov}_{\mathcal{G},\rho_t}\left(u-\mathbb{E}_{\rho_t}\right)\rho\\
&+\mathrm{Tr}\left(\mathrm{Cov}_{\rho_t,\mathcal{G}}\Gamma^{-1}\mathrm{Cov}_{\mathcal{G},\rho_t}\right)\rho du\,\\
=&\int_{\mathbb{R}^K}-\mathrm{Tr}\left(\mathrm{Cov}_{\rho_t,\mathcal{G}}\Gamma^{-1}\mathrm{Cov}_{\mathcal{G},\rho_t}\right)\rho du\leq 0\,,
\end{aligned}
\end{equation*}
which implies $\|\Cov_{\rho_t}\|_2\leq \|\Cov_{\rho_0}\|_2\leq C$. Furthermore, we also have
\[
\begin{aligned}
&\|\Cov_{\rho_t,\mathcal{G}}\|_2\leq \int_{\mathbb{R}^K}\|\left(u-\mathbb{E}_{\rho_t}\right)\left(\mathcal{G}(u)-\mathbb{E}_{\mathcal{G},\rho_t}\right)^\top\|_2\rho du\\
\leq &\int_{\mathbb{R}^K}\|\left(u-\mathbb{E}_{\rho_t}\right)\|_2\|\left(\mathcal{G}(u)-\mathbb{E}_{\mathcal{G},\rho_t}\right)\|_2\rho du\\
\leq &\left(\int_{\mathbb{R}^K}\|u-\mathbb{E}_{\rho_t}\|^2_2\rho du\right)^{\frac{1}{2}}\\
&\cdot\left(\int_{\mathbb{R}^K}\|\mathcal{G}(u)-\mathbb{E}_{\mathcal{G},\rho_t}\|_2^2\rho du\right)^{\frac{1}{2}}\\
\leq &\max(\|A\|_2,M)^{\frac{1}{2}}C\,,
\end{aligned}
\]
which proves \eqref{eqn:covbound}.\qed
\end{proof}

Such boundedness can be extended to higher moments:
\begin{lemma}\label{postervariance}
Let $\rho$ solve \eqref{eqn:muPDE} with initial condition $\mu_0$. If $\mu_0\in \mathcal{C}^2$ and has finite high moments, meaning for any $2\leq p<\infty$, there is a $C_{p,0}<\infty$ such that
\[
\int_{\mathbb{R}^L}|u|^p\mu_0(u)du=C_{p,0}<\infty\,.
\]
then under weakly nonlinear assumption \eqref{linear}, for any $2\leq p<\infty$, there is a constant $C_p<\infty$ such that: 
\begin{equation}\label{varianceestimation}
\begin{aligned}
&\int_{\mathbb{R}^L}|u-\mathbb{E}_{\rho_t}|^p\rho(t,u)du<C_p\,,\\
&\int_{\mathbb{R}^L}|u-u^\dagger|^p\rho(t,u)du<C_p\,,
\end{aligned}
\end{equation}
for all $0\leq t\leq 1$.
\end{lemma}

\begin{proof}
We first rewrite \eqref{eqn:muPDE} into the following form:
\[
\partial_t \rho=\nabla_u\cdot (F^\top(t,u)\rho)+\frac{1}{2}\mathrm{Tr}\left(D(t,u) D^\top(t,u)\mathcal{H}_u(\rho)\right)\,,
\]
where the flux term is
\[
F(t,u)=\mathrm{Cov}_{\rho_t,\mathcal{G}}(t)\Gamma^{-1}\left(y-\mathcal{G}(u)\right)
\]
and the hessian term is
\[
D(t,u)=\mathrm{Cov}_{\rho_t,\mathcal{G}}(t)\Gamma^{-\frac{1}{2}}\,.
\]
According to this definition and Lemma~\ref{postvarianceestimate}, $F(t,u)$ and $D(t,u)$ are Lipschitz and bounded respectively:
\begin{equation}\label{boundF}
\left|F(t,u_1)-F(t,u_2)\right|\leq C|u_1-u_2|,\quad |F(t,\vec{0})|\leq C\,,
\end{equation}
and
\begin{equation}\label{boundD}
|D(t,u)|\leq C\,,
\end{equation}
where $C$ is a constant independent of $t,u_1,u_2$.
\\
Consider the corresponding SDE to \eqref{eqn:muPDE}:
\[
dz_t=F(t,z_t)dt+D(t,z_t)dW_t\,
\]
with $z_0\sim \mu_0$, then $\int_{\mathbb{R}^L}|u|^p\rho(t,u)du=\EE|z_t|^p$ and it suffices to prove the boundedness of $\EE|z_t|^p$:
\begin{equation}\label{suffice}
    \int_{\mathbb{R}^L}|u|^p\rho(t,u)du=\EE|z_t|^p\leq C_p\,.
\end{equation}
Using It\^o's formula:
\[
\begin{aligned}
&\frac{d\EE|z_t|^{2k}}{dt}\\
\leq&2k\EE|z_t|^{2(k-1)}\left\langle z_t,F(t,z_t)\right\rangle \\
&+k\EE|z_t|^{2(k-1)}\mathrm{Tr}(D^\top(t,z_t) D(t,z_t))\\
&+2k(k-1)\EE|z_t|^{2(k-2)}\left\langle z_t,D(t,z_t) D^\top(t,z_t)z_t\right\rangle\\
\leq& C_{1,k}\EE|z_t|^{2k}+C_{2,k}\,,
\end{aligned}
\]
where $C_{1,k},C_{2,k}$ are constants only depending on $k$, and we use \eqref{boundF}-\eqref{boundD} and Young's inequality in the second inequality. For example:
\[
\begin{aligned}
&\EE|z_t|^{2(k-1)}\left\langle z_t,F(t,z_t)\right\rangle\\
\leq &\EE|z_t|^{2k-1}|F(t,z_t)|\\
\leq &\EE|z_t|^{2k-1}\left(C|z_t|+|F(t,\vec{0})|\right)\\
\leq &C\EE|z_t|^{2k}+C\EE|z_t|^{2k-1}\\
\leq &\left(C+\frac{2k-1}{2k}\right)\EE|z_t|^{2k}+\frac{C^{2k}}{2k}\,,
\end{aligned}
\]
where the last inequality comes from the Young's inequality:
\[
C\EE|z_t|^{2k-1}\leq \frac{2k-1}{2k}\EE|z_t|^{2k}+\frac{1}{2k}C^{2k}\,.
\]
Since
\[
\EE|z_0|^{2k}= \int_{\mathbb{R}^L}|u|^{2k}\mu_0(u)du<\infty\,,
\]
by Gr\"onwall's inequality, we finally obtain
\[
\EE|z_t|^{2k}\leq C'_{2k},\quad \forall 0\leq t\leq 1\,,
\]
which implies \eqref{suffice}. 
\\
Finally, \eqref{varianceestimation} follows from \eqref{suffice} and the boundedness of $u^\dagger$ and $\EE_{\rho_t}$.
\qed
\end{proof}

\subsection{$\{v^j\}$ and the Fokker-Planck-like equation}\label{sec:vj_FP}
The $\{v^j\}$ system is the ``bridge" we build to connect $\{u^j_t\}$ with the PDE. It follows the flow of the PDE:
\begin{equation}\label{eqn:mean_field}
\begin{aligned}
{dv^{j}_t}=&\mathrm{Cov}_{\rho_t,\mathcal{G}}\Gamma^{-1}\left(y-\mathcal{G}(v^j_t)\right)dt+\mathrm{Cov}_{\rho_t,\mathcal{G}}\Gamma^{-\frac{1}{2}}dW^{j}_t
\end{aligned}
\end{equation}
with $\mathrm{Cov}_{\rho_t,\mathcal{G}}$ determined by solution to \eqref{eqn:muPDE}. We denote its ensemble distribution
\[
M_{v}=\frac{1}{J}\sum^J_{j=1}\delta_{v^j_N}\,.
\]

It is a classical result that $W_2(M_v,\rho)\rightarrow 0$ in $J\to\infty$ limit in the expectation sense. Indeed, if the initial condition for this SDE system is consistent with $\mu_0$, meaning $\{v^j_0\}$ are drawn i.i.d. from the measure induced by $\mu_0$, then the ensemble distribution of $\{v^j_t\}$ is close to measure induced by $\rho_t$ for all finite time.

\begin{proposition}[Linking $\{v^j\}$ with Fokker-Planck-like PDE]\label{thm:vj_FP}
Let $\{v^j_{t}\}$ solve \eqref{eqn:mean_field} with $\{v^j_{0}\}$ drawn i.i.d. from the measure induced by $\mu_0$, and let $\rho(t,u)$ solve \eqref{eqn:muPDE} with initial condition $\mu_0$, then if $\mu_0\in\mathcal{C}^2$ and has finite high moments, then under the weakly nonlinear assumptions \eqref{linear}, there is a constant $C(t)$ independent of $J$ such that, 
\begin{equation}\label{eqn:vj_FP2}
\mathbb{E}\left(W_2(M_{v_t},\rho_t)\right)\leq C(t)
\left\{
\begin{aligned}
&J^{-\frac{1}{2}},\quad L<4\\
&J^{-\frac{1}{2}}\mathrm{log}(1+J),\quad L=4\\
&J^{-2/L},\quad L>4
\end{aligned}
\right.\,.
\end{equation}
for all $t<\infty$. Here $M_{v_t}$ is the ensemble distribution of $\{v^j_{t}\}$.
\end{proposition}

This is a straightforward consequence of the famous result by~\cite{Fournier2015}, and for the completeness we cite the theorem here:
\begin{theorem}[Theorem 1 in~\cite{Fournier2015}]\label{Fournier}
Let $\rho(u)$ be a probility density on $\mathbb{R}^L$ and let $p>0$. Assume that
\[
M_q(\rho):=\int_{\mathbb{R}^d}|x|^q\rho(dx)<\infty
\]
for some $q>p$. Consider an $i.i.d$ sequence $(X_k)_{k\geq1}$ of $\rho$-distributed random variables and, for $N\geq1$, define the empirical measure
\[
\rho_N:=\frac{1}{N}\sum^N_{k=1}\delta_{X_k}.
\]
There is a constant $C$ depending only on $p,q,L$ such that, for all $N\geq1$,

\begin{itemize}
\item[1.] If $p>L/2$ and $q\neq 2p$
\[
\mathbb{E}\left(W_p(\rho_N,\rho)\right)\leq N^{-\frac{1}{2}}+N^{-(q-p)/q}\,.
\]
\item[2.] If $p=L/2$ and $q\neq 2p$
\[
\mathbb{E}\left(W_p(\rho_N,\rho)\right)\leq N^{-\frac{1}{2}}\log(1+N)+N^{-(q-p)/q}\,.
\]
\item[3.] If $p\in(0,L/2)$ and $q\neq L/(L-p)$
\[
\mathbb{E}\left(W_p(\rho_N,\rho)\right)\leq N^{-p/L}+N^{-(q-p)/q}\,.
\]
\end{itemize}
\end{theorem}

To show Proposition~\ref{thm:vj_FP} one essentially only needs to show the boundedness of all moments of the particle system. This is given by the following Lemma \ref{prop:highmomentforpac2}. We simply choose a large enough $q$ to have the first terms in Theorem \ref{Fournier} being the dominant term that eliminates the second terms.

As a result of Lemma~\ref{postervariance}, we can also bound the high moments of $\{v^j\}$. This is indeed what we plan to do. In the lemma below we will show the boundedness of the moments of $\{v^j_t\}$, derived as a consequence of Lemma~\ref{postervariance}. Before starting the lemma, we first define
\begin{equation*}\label{qvqdef}
q^j_t=v^j_t-\overline{v}\,,
\end{equation*}
then we have:
\begin{lemma}\label{prop:highmomentforpac2}
Under conditions in Proposition \ref{thm:vj_FP}, for any fixed even number $2\leq p<\infty$ and large enough $J$, there exits a constant $C_p$ independent of $J$ such that for all $0\leq t\leq 1$:
\begin{align}\label{highmomentv}
&\mathbb{E}|v^j_t|^p\leq C_p,\ \mathbb{E}\left|q^j_t\right|^p\leq C_p,\quad\forall1\leq j\leq J\,,
\end{align}
and 
\begin{equation}\label{highmomentv4}
\begin{aligned}
&\left(\mathbb{E}\left\|\overline{v}-\EE_{\rho_t}\right\|^p_2\right)^{1/p}\lesssim J^{-\frac{1}{2}},\\
&\left(\mathbb{E}\left|\frac{1}{J}\sum^J_{j=1}|q^j_t|^2-\mathrm{Tr}(\Cov_{\rho_t})\right|^p\right)^{1/p}\lesssim J^{-\frac{1}{2}}\,.
\end{aligned}
\end{equation}
\begin{equation}\label{highmomentv3}
\left(\mathbb{E}\left\|\Cov_{v_t}-\Cov_{\rho_t}\right\|^p_2\right)^{1/p}\lesssim J^{-\frac{1}{2}}\,,
\end{equation}
\end{lemma}

\begin{proof}
Since $\{v^k_t\}$ are $i.i.d$ sampled from measure induced by $\rho(t,u)$, \eqref{highmomentv} is a direct result from \eqref{varianceestimation}. Now, we prove the first inequality in \eqref{highmomentv4}. Use Jensen's inequality, we have
\begin{equation}\label{firstinequalityproof}
\left(\EE\left|\overline{v}-\EE_{\rho_t}\right|^p\right)^{1/p}\leq \sum^L_{n=1}\left(\EE\left|\overline{\alpha}_n\right|^{p}\right)^{1/p}\,,
\end{equation}
where we denote
\[
\overline{\alpha}_{n}=\left(\overline{v}_t-\EE_{\rho(t)}\right)_n = \frac{1}{J}\sum\left(v^j_t-\EE_{\rho(t)}\right)_n = \frac{1}{J}\sum\alpha^j_n\,.
\]
The subscript $n$ means the $n$-th entry of the vector. It is easy to show, due to the fact that $\{v^j\}$ are i.i.d. that
\begin{equation}
\EE(\alpha^j_n)=0,\quad \EE|\alpha^j_{n}|^p<\infty\,.
\end{equation}
We also show in Appendix \ref{sec:momentbound} Lemma \ref{lemma:momentboundindependent} that
\begin{equation}\label{alphajmn2}
\mathbb{E}\left|\sum^J_{j=1}\alpha^j_{n}\right|^p\lesssim J^{p/2}\,,
\end{equation}
which implies
\begin{equation}\label{elementbound}
\begin{aligned}
\EE\left|\overline{\alpha}_n\right|^{p}&\leq \EE\left|\frac{1}{J}\sum^J_{j=1}\alpha^j_n\right|^{p}\lesssim O(J^{-p/2})\,.
\end{aligned}
\end{equation}
Plugging~\eqref{elementbound} into \eqref{firstinequalityproof}, we prove the first inequality of \eqref{highmomentv4}. To show the second inequality in \eqref{highmomentv4} we note:
\[
\begin{aligned}
&\left|\frac{1}{J}\sum^J_{j=1}|q^j_t|^2-\mathrm{Tr}(\Cov_{\rho_t})\right|=\left|\mathrm{Tr}\left(\Cov_{v_t}-\Cov_{\rho_t}\right)\right|\\
\leq &L\left\|\Cov_{v_t}-\Cov_{\rho_t}\right\|_2
\end{aligned}
\]
Therefore it would be a direct result from~\eqref{highmomentv3}.
\\
To show~\eqref{highmomentv3}, we write $\Cov_{v_t}$ as
\[
\Cov_{v_t}=\frac{1}{J}\left(\sum^{J}_{j=1}v^j_t\otimes v^j_t\right)-\overline{v}\otimes \overline{v}
\]
meaning:
\begin{equation}\label{seperate}
\begin{aligned}
&\left(\mathbb{E}\left\|\Cov_{v_t}-\Cov_{\rho_t}\right\|^p_2\right)^{1/p}\\
\leq&\left(\mathbb{E}\left\|\frac{1}{J}\left(\sum^{J}_{j=1}v^j_t\otimes v^j_t\right)-\mathbb{E}_{\rho(t)}(v\otimes v)\right\|^p_2\right)^{1/p}\\
&\ +\left(\EE\left\|\overline{v}\otimes \overline{v}-\EE_{\rho_t}\otimes \EE_{\rho_t}\right\|^p_2\right)^{1/p}\,.
\end{aligned}
\end{equation}
We show below that both terms are of order $J^{-1/2}$. To show this for the first term, let 
\[
W=\sum^{J}_{j=1}\left(v^j_t\otimes v^j_t-\mathbb{E}_{\rho(t)}(v\otimes v)\right)=\sum_j w^j\,,
\]
then the first term becomes
\[
\begin{aligned}
&\left(\mathbb{E}\left\|\frac{1}{J}W\right\|^p_2\right)^{1/p}\leq \left(\mathbb{E}\left\|\frac{1}{J}W\right\|^p_F\right)^{1/p}\\
\lesssim&\, \sum^L_{m,n=1} \left(\mathbb{E}|W_{m,n}/J|^{p}\right)^{1/p}\\
=&\, \sum^L_{m,n=1} \frac{1}{J^{1/2}}\left(\mathbb{E}|W_{m,n}/\sqrt{J}|^{p}\right)^{1/p}\,,
\end{aligned}
\]
where $W_{m,n}$ means the $(m,n)^{th}$ entry of matrix. Similar to before, for each $m,n$, we have
\begin{equation}
\EE(w^j_{m,n})=0,\quad \EE|w^j_{m,n}|^p<\infty\,,
\end{equation}
and by Appendix \ref{sec:momentbound} Lemma \ref{lemma:momentboundindependent}, we have
\begin{equation}\label{wjmn2}
\mathbb{E}\left|\sum^J_{j=1}w^j_{m,n}\right|^p\lesssim J^{p/2}\,,
\end{equation}
which implies
\[
\mathbb{E}{|W/\sqrt{J}|_{m,n}^{p}}=\mathbb{E}\left|\frac{\sum^{J}_{j=1}w^j_{m,n}}{\sqrt{J}}\right|^{p}\sim O(1)
\]
which makes the first term $J^{-1/2}$.
For the second term in \eqref{seperate}, we have
\begin{equation}\label{seperatesecondterm}
\begin{aligned}
&\left(\EE\left\|\overline{v}\otimes \overline{v}-\EE_{\rho_t}\otimes \EE_{\rho_t}\right\|^p_2\right)^{1/p}\\
\leq &\left(\EE\left\|\left(\overline{v}-\EE_{\rho_t}\right)\otimes \overline{v}\right\|^p_2\right)^{1/p}+\left(\EE\left\|\EE_{\rho_t}\otimes \left(\overline{v}-\EE_{\rho_t}\right)\right\|^p_2\right)^{1/p}\,,
\end{aligned}
\end{equation}
The first term of \eqref{seperatesecondterm} can be bounded by
\[
\begin{aligned}
&\left(\EE\left\|\left(\overline{v}-\EE_{\rho_t}\right)\otimes \overline{v}\right\|^p_2\right)^{1/p}\\
\leq &\left(\EE\left\|\overline{v}-\EE_{\rho_t}\right\|^p_2\left\|\overline{v}\right\|^p_2\right)^{1/p}\\
\stackrel{(I)}{\leq}& \left(\EE\left\|\overline{v}-\EE_{\rho_t}\right\|^{2p}_2\right)^{1/2p}\left(\EE\left\|\overline{v}\right\|^{2p}_2\right)^{1/2p}\\
\stackrel{(II)}{\lesssim} &J^{-1/2}\,,
\end{aligned}
\]
where we use H\"older's inequality in $(I)$ and \eqref{highmomentv} and first inequality in \eqref{highmomentv4} in $(II)$. Similarly, second term of \eqref{seperatesecondterm} can also be bounded by
\[
\left(\EE\left\|\EE_{\rho_t}\otimes \left(\overline{v}-\EE_{\rho_t}\right)\right\|^p_2\right)^{1/p} \lesssim J^{-1/2}\,.
\]
Plug these two inequalities into \eqref{seperatesecondterm}, we have
\[
\left(\EE\left\|\overline{v}\otimes \overline{v}-\EE_{\rho_t}\otimes \EE_{\rho_t}\right\|^p_2\right)^{1/p}\lesssim J^{-1/2}\,.
\]
In conclusion, we finally obtain~\eqref{highmomentv3}.

\qed
\end{proof}

\section{Proof of Theorem~\ref{thm:mean_field} Part II, and Theorem~\ref{thm:moments}}\label{sec:mean_field_2}
We are now left with the task to show the closeness of $\{u^j_t\}$ and $\{v^j_t\}$. The two systems are governed by the SDE~\eqref{eqn:SDE}, and~\eqref{eqn:mean_field}.

The precise statement is the following:
\begin{proposition}\label{thm:vj_uj}[Linking $\{u^j\}$ with $\{v^j\}$]
Let $\{v^j_t\}^J_{j=1}$ solve~\eqref{eqn:mean_field}  and $\{u^j_t\}^J_{j=1}$ solve~\eqref{eqn:SDE}, with the same initial data i.i.d drawn from the measure induced by $\mu_0$. If $\mu_0\in\mathcal{C}^2$ and has finite high moments, then under weakly nonlinear assumptions \eqref{linear}, the two SDE systems are close in the following sense: for any $0<\epsilon<\frac{1}{4}$, there is a constant $0<C_\epsilon<\infty$ independent of $J$ and $t$ such that for any $0\leq t\leq 1$
\begin{equation}\label{Expectationdifference}
\frac{1}{J}\sum^J_{j=1}\mathbb{E}|u^j_{t}-v^j_{t}|^2\leq {C_\epsilon}{J^{-1+\epsilon}}\,.
\end{equation}
Furthermore, denote $M_{v_t}$ and $M_{u_t}$ the ensemble distributions of $\{v^j_t\}$ and $\{u^j_t\}$ respectively, then
\begin{equation}\label{eqn:W2_estimate}
\begin{aligned}
&\mathbb{E}\left(W_2(M_{v_t},M_{u_t}\right))\\
\leq&\left(\frac{1}{J}\sum^J_{j=1}\mathbb{E}|u^j_{t}-v^j_{t}|^2\right)^{\frac{1}{2}}\leq {C_\epsilon}{J^{-\frac{1}{2}+\epsilon}}\,.
\end{aligned}
\end{equation}
\end{proposition}

This proposition states that the two particle systems are close for big $J$. Combined with Proposition~\ref{thm:vj_FP}, it is straightforward to show Theorem~\ref{thm:mean_field}.
\begin{proof}[Proof of Theorem~\ref{thm:mean_field}]
Considering~\eqref{eqn:vj_FP2} and~\eqref{eqn:W2_estimate}, by triangle inequality, for any $0\leq t\leq 1$, one has:
\[
\begin{aligned}
&\mathbb{E}\left(W_2(M_{u_t},\rho(t,u))\right)\\
\leq &\mathbb{E}\left(W_2(M_{u_t},M_{v_t})\right)+\mathbb{E}\left(W_2(M_{v_t},\rho(t,u))\right)\\
\leq &C_\epsilon
\left\{
\begin{aligned}
&J^{-\frac{1}{2}+\epsilon},\quad L\leq4\\
&J^{-2/L},\quad L>4\\
\end{aligned}
\right.\,,
\end{aligned}
\]
which finishes the proof.\qed
\end{proof}

The proof for Theorem~\ref{thm:moments} is also straightforward.
\begin{proof}[Proof of Theorem~\ref{thm:moments}]
Using triangle inequality to the left hand side of \eqref{weakconvergence}, we have
\begin{equation}\label{uf1}
\begin{aligned}
&\left(\EE\left|\int f(u)\left[M_{u_t}-\rho(t,u)\right] du\right|^2\right)^{\frac{1}{2}}\\
\leq &\left(\EE\left|\int f(u)\left[M_{u_t}-M_{v_t}\right] du\right|^2\right)^{\frac{1}{2}}\\
&+\left(\EE\left|\int f(u)\left[M_{v_t}-\rho(t,u)\right] du\right|^2\right)^{\frac{1}{2}}\,.
\end{aligned}
\end{equation}
We bound both terms:
\begin{itemize}
\item Expand the first term: we have
\begin{equation}\label{uf2}
\begin{aligned}
&\EE\left|\int f(u)\left[M_{u_t}-M_{v_t}\right] du\right|^2\\
=&\EE\left|\frac{1}{J}\sum^J_{j=1}f(u^j_t-v^j_t)\right|^2\\
\leq&\frac{l^2}{J^2}\EE\left(\sum^J_{j=1}|u^j_t-v^j_t|^2\right)\\
\leq&C_\epsilon L^2J^{-1+\epsilon}\,,
\end{aligned}
\end{equation}
where in the first inequality we use $f$ is $l$-Lipshitz and H\"older's inequality and in the second inequality we use Proposition \ref{thm:vj_uj} \eqref{Expectationdifference}.
\item Consider the second term, we have
\[
\begin{aligned}
&\EE\left|\int f(u)\left[M_{v_t}-\rho(t,u)\right] du\right|^2\\
=&\EE\left|\frac{1}{J}\sum^J_{j=1}f(v^j_t)-\EE_{\rho_t}(f)\right|^2\\
=&\frac{1}{J^2}\sum^J_{j=1}\EE\left|f(v^j_t)-\EE_{\rho_t}(f)\right|^2\\
\leq&\mathrm{Cov}_{\rho_t,f}J^{-1}\,,
\end{aligned}
\]
where in the second equality we use $v^j_t\sim \rho(t,u)$ are independent and $\mathrm{Cov}_{\rho_t,f}$ is same as covariance of $f$.
\\
Since $f$ is $l$-Lipschitz and $\rho$ has finite second moment, there is a constant $C(l,f(\vec{0}))$ such that 
\[
\mathrm{Cov}_{\rho_t,f}\leq C(l,f(\vec{0}))\,.
\]
Therefore, we have
\begin{equation}\label{vf1}
\EE\left|\int f(u)\left[M_{v_t}-\rho(t,u)\right] du\right|^2\leq C(l,f(\vec{0}))J^{-1}\,.
\end{equation}
\end{itemize}
Combine the two terms into \eqref{uf1}, we proves \eqref{weakconvergence} with the constant depending on $\epsilon$, $l$ and $f(\vec{0})$.\qed
\end{proof}

In the following subsections, we first provide some a-priori estimate, and prove Proposition~\ref{thm:vj_uj} using the bootstrapping method.

\subsection{Some a-priori estimates}
We mainly show the higher moments of $\{u^j\}$ are bounded.

First, we present a lemma similar to proof of Theorem 4.5 in \cite{DCPS}. For convenience, denote
\begin{align*}
&e^j(t)=u^j(t)-\overline{u}(t)\,,\quad \textbf{e}^j(t)=\Gamma^{-\frac{1}{2}}Ae^j(t)\,,\\
&\tu^j(t)=\Gamma^{-\frac{1}{2}}Au^j(t)\,,\\
&\,\tr^j(t)=\Gamma^{-\frac{1}{2}}\left[\Rm(u^j(t))-\frac{1}{J}\sum^J_{j=1}\Rm(u^j(t))\right]\,.
\end{align*}
then:
\begin{lemma}\label{anotherestimate}
Denote
\begin{equation}
\mathrm{V}_p(e(t)):=\mathbb{E}\left(\sum^K_{m=1}\left(\frac{1}{J}\sum^J_{j=1}\left|e^j_m(t)\right|^2\right)^{p/2}\right)
\end{equation}
for some $p\geq2$. Then under conditions of Proposition~\ref{thm:vj_uj}, for every $p$, there is a constant $J_p$ such that for any $J>J_p$ and $0\leq t\leq 1$
\begin{equation}\label{vp}
\mathrm{V}_p(e(t))\leq C_p\,,
\end{equation}
where $C_p$ is a constant independent of $J$ and $t$. Moreover, $J_2=0$. Here, $e^j_m$ is the $m$-th component of $e^j$.
\end{lemma}
\begin{proof} Without loss of generality, assume $u^\dagger=\vec{0}$. When $t=0$, since $\mu_0$ has finite high moments, we can find a bound for $\mathrm{V}_p(e(0))$ independent of $J$. Let 
\[
\mathrm{W}_p(e(t))=\sum^K_{m=1}\left(\frac{1}{J}\sum^J_{j=1}\left|e^j_m(t)\right|^2\right)^{p/2}\,,
\]
then we have
\[
\begin{aligned}
de^j_m=&-\frac{1}{J}\sum^J_{k=1}e^k_m\left\langle \textbf{e}^k,\textbf{e}^j\right\rangle dt-\frac{1}{J}\sum^J_{k=1}e^k_m\left\langle \tr^k,\tr^j\right\rangle dt\\
&+\frac{1}{J}\sum^J_{k=1}e^k_m\left\langle \textbf{e}^k,d\left(W^j-\overline{W}\right)\right\rangle\\
&+\frac{1}{J}\sum^J_{k=1}e^k_m\left\langle \tr^k,d\left(W^j-\overline{W}\right)\right\rangle
\end{aligned}\,
\]
and

\begin{equation}\label{eqn:WPE}
\begin{aligned}
d\mathrm{W}_p(e)=&\sum^K_{m=1}\sum^{J}_{j=1}\frac{\partial \mathrm{W}_p}{\partial e^j_m}de^j_m\\
&+\frac{1}{2}\sum^K_{m=1}\sum^{J}_{j,j'=1}de^j_m\frac{\partial^2 \mathrm{W}_p}{\partial e^j_m\partial e^{j'}_{m}}de^{j'}_{m}\,.
\end{aligned}
\end{equation}
Let
\begin{equation}\label{eqn:E}
\mathcal{E}=\sum^K_{m=1}\left(\frac{1}{J}\sum^J_{j=1}|e^j_m|^2\right)^{\frac{p}{2}-1}\sum^K_{n=1}\left(\sum^J_{k=1}e^k_m\textbf{e}^k_n\right)^2\,,
\end{equation}
\begin{equation}\label{eqn:R}
\mathcal{R}=\sum^K_{m=1}\left(\frac{1}{J}\sum^J_{j=1}|e^j_m|^2\right)^{\frac{p}{2}-1}\sum^K_{n=1}\left(\sum^J_{k=1}e^k_m\textbf{r}^k_n\right)^2\,,
\end{equation}
\[
\mathcal{F}=\sum^K_{m=1}\left(\frac{1}{J}\sum^J_{j=1}|e^j_m|^2\right)^{\frac{p}{2}-1}\sum^K_{n=1}\left(\sum^J_{k=1}e^k_m(\textbf{e}^k_n+\textbf{r}^k_n)\right)^2\,.
\]
Using Young's inequality: $(a+b)^2\leq (1+\epsilon) a^2+(1+1/\epsilon)b^2$ for any $\epsilon>0$, we have
\begin{equation}\label{OER}
\mathcal{F}\leq (1+\epsilon)\mathcal{E}+(1+1/\epsilon)\mathcal{R}\,.
\end{equation}
Similar to \cite{DCPS} (B.1), taking expectation on the first part of \eqref{eqn:WPE} gives us:
\begin{equation}\label{eqn:WPE1}
\EE\left(\sum^K_{m=1}\sum^{J}_{j=1}\frac{\partial \mathrm{W}_p}{\partial e^j_m}de^j_m\right)=-\frac{p
}{J^2}\EE(\mathcal{E}+\mathcal{R})
\end{equation}
and the second part of \eqref{eqn:WPE} give us:
\begin{equation}\label{eqn:WPE2}
\begin{aligned}
&\EE\left(\frac{1}{2}\sum^K_{m=1}\sum^{J}_{j,j'=1}de^j_m\frac{\partial^2 \mathrm{W}_p}{\partial e^j_m\partial e^{j'}_{m}}de^{j'}_{m}\right)\leq C\mathbb{E}(\mathcal{F})\\
\leq&C(1+\epsilon)\EE(\mathcal{E})+C(1+1/\epsilon)\EE(\mathcal{R})
\end{aligned}
\end{equation}
where $C=\frac{p}{J^2}\left(\frac{(p-2+J)(J-1)}{2J^2}+\frac{(p-2)}{2J^2}\right)$ and in the last inequality we use \eqref{OER} with $\epsilon>0$.
\\
Plug \eqref{eqn:E} and \eqref{eqn:R} into \eqref{eqn:WPE1} and \eqref{eqn:WPE2} with $\epsilon=\frac{1}{2}$, then the expectation of $\mathrm{W}_p$ is given by
\begin{equation}\label{VpET}
\begin{aligned}
&\frac{d\mathrm{V}_p(e)}{dt}=\frac{d\EE\mathrm{W}_p(e)}{dt}\leq -C_1\EE(\mathcal{E})+C_2\EE(\mathcal{R})\\
= &-C_1\mathbb{E}\left[\sum^K_{m=1}\left(\sum^J_{j=1}|e^j_m|^2\right)^{\frac{p}{2}-1}\sum^K_{n=1}\left(\sum^J_{k=1}e^k_m\textbf{e}^k_n\right)^2\right]\,\\
&+C_2\mathbb{E}\left[\sum^K_{m=1}\left(\sum^J_{j=1}|e^j_m|^2\right)^{\frac{p}{2}-1}\sum^K_{n=1}\left(\sum^J_{k=1}e^k_m\textbf{r}^k_n\right)^2\right]\\
\leq &C_3\mathbb{E}\left[\sum^K_{m=1}\left(\sum^J_{j=1}|e^j_m|^2\right)^{\frac{p}{2}}\right]\\
=&C_4\mathrm{V}_p(e)
\end{aligned}\,,
\end{equation}
where  
\[
C_1=\frac{p}{J^{1+p/2}}\left(1-\frac{3(p-2+J)(J-1)}{4J^2}-\frac{3(p-2)}{4J^2}\right)\,,
\]
\[
C_2=-\frac{p}{J^{1+p/2}}\left(1-\frac{3(p-2+J)(J-1)}{2J^2}-\frac{3(p-2)}{2J^2}\right)\,,
\]
\[
C_3=4\|\Gamma^{-\frac{1}{2}}\|^2_2M^2J\times C_2\,,
\]
\[
C_4=C_3\times J^{p/2}\sim O(1)\,.
\]
From the second to the third inequality, we delete the first term since it is always negative. We also used the following:
\[
\begin{aligned}
\sum^K_{n=1}\left(\sum^J_{k=1}e^k_m\textbf{r}^k_n\right)^2&\leq \left(\sum^J_{k=1}|e^k_m|^2\right)\left(\sum^K_{n=1}\sum^J_{k=1}|\textbf{r}^k_n|^2\right)\\
&\leq 4\|\Gamma^{-\frac{1}{2}}\|^2_2M^2J\left(\sum^J_{k=1}|e^k_m|^2\right)
\end{aligned}
\]
to obtain the formula for $C_3$. Note that there is a number $J_p$ such that when $J>J_p$, the constants are all positive. Note that according to the formula of $C_{1}$ and $C_{2}$, $J_2=0$. Since $V_p(e(0))$ is bound, by the Gr\"onwall inequality, \eqref{VpET} implies \eqref{vp}.
\qed
\end{proof}

\begin{lemma}\label{prop:highmomentbound}
Under conditions of Proposition \ref{thm:vj_uj}, for any $2\leq p<\infty$ and large enough $J$ (larger than $J_p$ as defined in Lemma \ref{anotherestimate}), $p$-th moment of particles $\{u^j_t\}^J_{j=1}$ are uniformly bounded for finite time, namely there is a constant $C_p>0$ independent of $J$ and $t$ such that for all $0\leq t\leq 1$ and $1\leq j\leq J$
\begin{equation}\label{highmomentu}
\mathbb{E}|u^j_t|^p\leq C_p\,,\ \left(\mathbb{E}\left\|\Cov_{u_t}-\Cov_{\rho_t}\right\|^p_2\right)^{1/p}\leq C_p\,.
\end{equation}
Furthermore, 
\[
\mathbb{E}\left|u^j_t-\bar{u}_t\right|^p\leq C_p\,,\ \mathbb{E}\left|u^j_t-u^\dagger\right|^p\leq C_p\,.
\]
\end{lemma}

We note that the linear case with $p=2$ was studied in~\cite{DCPS} (Proposition 4.11 and 5.1). This will not be enough for our use in the later section since our analysis crucially depends on the boundedness of higher moments. We leave the proof in Appendix \ref{appendixhighmoment}.

Combining Lemma~\ref{prop:highmomentforpac2} and Lemma~\ref{prop:highmomentbound}, using triangle inequality we have:
\begin{corollary}\label{cor:prebound}
Under conditions of Proposition \ref{thm:vj_uj}, for any $2\leq p<\infty$ and large enough $J$ (larger than $J_p$ as defined in Lemma \ref{anotherestimate}), we have a constant $C_p$ independent of $J$ such that for all $1\leq j\leq J$ and $0\leq t\leq 1$
\begin{equation}\label{uvtrivialbound}
\mathbb{E}|u^j_t-v^j_t|^p=\mathrm{E}|u^1_t-v^1_t|^p\leq C_p\,.
\end{equation}
\end{corollary}
\subsection{Proof of Proposition~\ref{thm:vj_uj}}
To show Proposition~\ref{thm:vj_uj}, we first unify the notations. Without loss of generality, we let $u^\dagger=\vec{0}$. We further use the following notations for conciseness. Let
\[
x^j_t=u^j_t-v^j_t\,,\quad p^j_t=x^j_t-\overline{x}_t\,,
\]
and denote (call them observables)
\[
\begin{aligned}
&\tx^j_t=\Gamma^{-\frac{1}{2}}Ax^j_t\,,\quad \tu^j_t=\Gamma^{-\frac{1}{2}}Au^j_t\,,\\
&\tv^j_t=\Gamma^{-\frac{1}{2}}Av^j_t\,,\quad \tp^j_t=\Gamma^{-\frac{1}{2}}A(x^j_t-\overline{x}_t)\,,\\
&\tq^j_t=\Gamma^{-\frac{1}{2}}A(v^j_t-\overline{v}_t)\,.
\end{aligned}
\]
We also use notation $A\lesssim O(J^\alpha)$ to mean that there is a constant $C$ independent of $J$ so that $A\leq CJ^{\alpha}$.

To prove the theorem amounts to tracing the evolution of $\mathbb{E}|x^j_t|^2$ as a function of time and $J$. For that we use the bootstrapping argument, namely, we assume $\mathbb{E}|x^j_t|^2$ decays in $J$ with certain rate (could be $0$, as have already suggested in Lemma~\ref{prop:highmomentbound} and Corollary~\ref{cor:prebound}), then by following the flow of the SDE we can show the rate can be tightened till a threshold is achieved. This threshold is exactly the rate one needs to prove in Proposition~\ref{thm:vj_uj}.

The tightening procedure is discussed in Lemma~\ref{lem:bound_tx} and Lemma~\ref{lem:bound_x} respectively for observables $\tx^j_t$, and the true error $x^j_t$. The proof of the proposition is an immediate consequence.

In the proofs we will constantly use the fact that
\[
\mathbb{E}|\tp^j_t|^2=\mathbb{E}|\tp^1_t|^2\,,\ \mathbb{E}|\tx^j_t|^2=\mathbb{E}|\tx^1_t|^2
\]
for all $0\leq t\leq 1$ and $1\leq j\leq J$. When the context is clear, we also omit subscript $t$ for the simplicity of the notation. 

We first show $|\overline{x}|^2,|p^j|^2,|\overline{\tx}|^2,|\tp^j|^2$ can be bounded by $|x^j|^2$. 
\begin{lemma}\label{lem:tpoverlinex}
For any $0\leq\alpha<1$, and $0\leq t\leq 1$, with the definition above, if one has:
\begin{equation}\label{preboundconditionforx0}
\mathbb{E}|x^j|^2\lesssim O\left(J^{-\alpha}\right)\,
\end{equation}
for all $1\leq j\leq J$, then
\begin{equation}\label{tpoverlinex2}
\mathbb{E}|\tx^j|^2\lesssim O\left(J^{-\alpha}\right)\,
\end{equation}
and
\begin{equation}\label{tpoverlinex1}
\mathbb{E}|p^j|^2\lesssim O\left(J^{-\alpha}\right)\,,\quad \mathbb{E}|\tp^j|^2\lesssim O\left(J^{-\alpha}\right)\,
\end{equation}
for all $1\leq j\leq J$.
\end{lemma}
\begin{proof} Due to~\eqref{preboundconditionforx0}, we first have for all $j$,
\[
\begin{aligned}
\left(\mathbb{E}|p^j|^2\right)^{\frac{1}{2}}&=\left(\mathbb{E}\left|\frac{J-1}{J}x^j-\frac{1}{J}\sum^J_{k\neq j}x^k\right|^2\right)^{\frac{1}{2}}\\
&\leq 2\left(\mathbb{E}|x^1|^2\right)^{\frac{1}{2}}\lesssim O\left(J^{-\frac{\alpha}{2}}\right)\,
\end{aligned}
\]
and 
\[
\left(\mathbb{E}|\overline{x}|^2\right)^{\frac{1}{2}}\leq \frac{1}{J}\sum^J_{j=1}\left(\EE|x^j|^2\right)^{\frac{1}{2}}\lesssim O\left(J^{-\frac{\alpha}{2}}\right)\,,
\]
which implies first inequality in \eqref{tpoverlinex1}. Then we also have an estimate for $\tx^j$:
\[
\mathbb{E}|\tx^j|^2\lesssim \|\Gamma^{-\frac{1}{2}}A\|_2\mathbb{E}|x^j|^2\lesssim O\left(J^{-\alpha}\right)\,,
\]
which implies \eqref{tpoverlinex2} and it also leads to
\[
\begin{aligned}
\left(\mathbb{E}|\tp^j|^2\right)^{\frac{1}{2}}&=\left(\mathbb{E}\left|\frac{J-1}{J}\tx^j-\frac{1}{J}\sum^J_{k\neq j}\tx^k\right|^2\right)^{\frac{1}{2}}\\
&\leq 2\left(\mathbb{E}|\tx^1|^2\right)^{\frac{1}{2}}\lesssim O\left(J^{-\frac{\alpha}{2}}\right)\,
\end{aligned}
\]
and 
\[
\left(\mathbb{E}|\overline{\tx}|^2\right)^{\frac{1}{2}}\leq \frac{1}{J}\sum^J_{j=1}\left(\EE|\tx^j|^2\right)^{\frac{1}{2}}\lesssim O\left(J^{-\frac{\alpha}{2}}\right)\,.
\]
This finishes the proof.\qed
\end{proof}

Then we show if we already have an a-priori estimate for $\{x^j\}$, we can have a better control for $\{\tx^j\}$.

\begin{lemma}\label{lem:bound_tx}
For any $0\leq\alpha<1$, and $0\leq t\leq 1$, if one has:
\begin{equation}\label{preboundconditionforx}
\mathbb{E}|x^j|^2\lesssim O\left(J^{-\alpha}\right)\,,
\end{equation}
for all $j$, then for any $0<\epsilon<\frac{1}{4}$ , there is $C_\epsilon<\infty$ independent of $J$ and $t$ such that
\begin{equation*}\label{pestimation}
\begin{aligned}
&\mathbb{E}|\tp^j|^2=\mathbb{E}\left|\tx^j-\frac{1}{J}\sum^J_{k}\tx^k\right|^2\leq C_\epsilon J^{-\frac{1}{2}-\frac{\alpha}{2}+\epsilon}\,,\\
&\mathbb{E}|\tx^j|^2\leq C_\epsilon J^{-\frac{1}{2}-\frac{\alpha}{2}+\epsilon}\,.
\end{aligned}
\end{equation*}
for all $j$. Note for any $\alpha<1$, we can choose $\epsilon<1-\alpha$ to make $\frac{1}{2}+\frac{\alpha}{2}-\epsilon>\alpha$.
\end{lemma}

\begin{proof}
Firstly, by Lemma \ref{lem:tpoverlinex} equations \eqref{tpoverlinex2},\eqref{tpoverlinex1} we have a rough estimate for $\tx^j,\tp^j,\overline{\tx}$
\begin{equation}\label{pestimationforp}
\begin{aligned}
&\mathbb{E}|\tx^j|^2\lesssim O\left(J^{-\alpha}\right),\ \mathbb{E}|\tp^j|^2\lesssim O\left(J^{-\alpha}\right),\\
&\mathbb{E}|\overline{\tx}|^2\lesssim O\left(J^{-\alpha}\right)\,.
\end{aligned}
\end{equation}
Apply $\Gamma^{-\frac{1}{2}}A$ on both sides of~\eqref{eqn:SDE} and~\eqref{eqn:mean_field}, we have the evolution of the observables:
\begin{equation}\label{eqn:star1}
\begin{aligned}
{d\tu^{j}}=&-\mathrm{Cov}_{\tu_t,\tu_t}\tu^j dt+\mathrm{Cov}_{\tu_t,\tu_t}dW^{j}_t\\
&+\mathrm{Cov}_{\tu_t,\Rm}\Gamma^{-1}\left(r-\Rm(u^j)\right)dt\\
&+\mathrm{Cov}_{\tu_t,\Rm}\Gamma^{-\frac{1}{2}}dW^{j}_t
\end{aligned}\,
\end{equation}
and
\begin{equation}\label{eqn:star2}
\begin{aligned}
d\tv^{j}=&-\Gamma^{-\frac{1}{2}}A\mathrm{Cov}_{\rho_t}A^\top \Gamma^{-\frac{1}{2}}\tv^{j}dt\\
&+\Gamma^{-\frac{1}{2}}A\mathrm{Cov}_{\rho_t}A^\top \Gamma^{-\frac{1}{2}}dW^{j}_t\\
&+\Gamma^{-\frac{1}{2}}A\mathrm{Cov}_{\rho_t,\Rm}\Gamma^{-1}\left(r-\Rm(v^j)\right)dt\\
&+\Gamma^{-\frac{1}{2}}A\mathrm{Cov}_{\rho_t,\Rm}\Gamma^{-\frac{1}{2}}dW^{j}_t\,.
\end{aligned}
\end{equation}
Subtracting the two equations we can derive the evolution of $\tx^j$. With the calculation shown in Supp. A, for any $0<\epsilon<\frac{1}{4}$, there is a $J^*_\epsilon>0$ such that for $J>J^*_\epsilon$ and $0\leq t\leq 1$
\begin{equation}\label{eqn:firstevolution}
\begin{aligned}
&\frac{d\frac{1}{J}\sum^J_{j=1}\EE|\tx^j|^2}{dt}\\
\leq&\ C_\epsilon J^{-\frac{1}{4}}\left(\left(\EE|\tx^1|^2\right)^{1-\epsilon}+\left(\EE|\overline{\tx}|^2\right)^{1-\epsilon}+\left(\EE|\tp^1|^2\right)^{1-\epsilon}\right)\\
&+C\left(\EE|\tx^1|^2+\EE|\tp^1|^2\right)\\
&+C_\epsilon J^{-\frac{1}{2}}\left(\left(\mathbb{E}\left|\tx^1\right|^2\right)^{\frac{2-\epsilon}{4}}+\left(\mathbb{E}\left|\tp^1\right|^2\right)^{\frac{2-\epsilon}{4}}\right)\\
&+C_\epsilon J^{-\frac{1}{2}}\EE|\tx^1|^2+C_\epsilon J^{-1}\,,
\end{aligned}
\end{equation}
where $C_\epsilon$ is a constant independent of $J$ and $t$. This leads to, plugging in~\eqref{preboundconditionforx} and~\eqref{pestimationforp}:
\begin{equation*}
\begin{aligned}
\frac{d\mathbb{E}|\tx^1|^2}{dt}=&\frac{1}{J}\sum^J_{j=1}\frac{d\mathbb{E}|\tx^j|^2}{dt}\\
\leq&\ C_\epsilon\EE|\tx^1|^2+C_\epsilon J^{-\frac{1}{4}}\left(\EE|\tx^1|^2\right)^{1-\epsilon}\\
&\ +C_\epsilon J^{-\frac{1}{2}-\frac{\alpha}{2}+\frac{\alpha\epsilon}{4}}\,.
\end{aligned}
\end{equation*}
Define $\mathsf{X}^\beta=\mathbb{E}J^\beta|\tx^1|^{2}$, the equation rewrites as
\[
\begin{aligned}
\frac{d\mathsf{X}^\beta}{dt}\leq C_\epsilon\mathsf{X}^\beta&+C_\epsilon J^{-\frac{1}{4}+\epsilon\beta}\left(\mathsf{X}^\beta\right)^{1-\epsilon}\\
&+C_\epsilon J^{-\frac{1}{2}-\frac{\alpha}{2}+\frac{\alpha\epsilon}{4}+\beta}\,.
\end{aligned}
\]
Because $\mathsf{X}^\beta(0)=0$, this implies
\begin{equation}\label{Gronexample1}
\|\mathsf{X}^\beta\|_{L^\infty}\lesssim \max\left\{O(1),J^{-\frac{1}{4}+\epsilon\beta},J^{-\frac{1}{2}-\frac{\alpha}{2}+\frac{\alpha\epsilon}{4}+\beta}\right\}\,,
\end{equation}
for $J>J^*_\epsilon$. For $J\leq J^*_\epsilon$, according to Corollary \ref{cor:prebound}, one still has
\[
\|\mathsf{X}^\beta\|_{L^\infty}\leq (J^{\ast}_\epsilon)^\beta\sup_{0\leq t\leq 1}\EE|\tx^1_t|^2\leq (J^{\ast}_\epsilon)^\beta C\lesssim O(1)\,.
\]
This can be absorbed in \eqref{Gronexample1} and \eqref{Gronexample1} is true for any $J>0$.

This finally suggests, if we choose $\beta=\frac{1}{2}+\frac{\alpha}{2}-\frac{\alpha\epsilon}{4}$, then
\[
\mathbb{E}|\tx^j|^2=\mathbb{E}|\tx^1|^{2}_2\lesssim O\left(J^{-\frac{1}{2}-\frac{\alpha}{2}+\frac{\alpha\epsilon}{4}}\right)\,,
\]
and
\[
\mathbb{E}|\tp^j|^2\leq 2\mathbb{E}|\tx^j|^2=2\mathbb{E}|\tx^1|^{2}_2\lesssim O\left(J^{-\frac{1}{2}-\frac{\alpha}{2}+\frac{\alpha\epsilon}{4}}\right)\,,
\]
for any $0<\epsilon<\frac{1}{4}$ and $1\leq j\leq J$. The $O$ notation includes a constant $C_\epsilon$ that has $\epsilon$ dependence.\qed
\end{proof}

This allows us to give a tighter bound for $\mathbb{E}|x^j|^2$:
\begin{lemma}\label{lem:bound_x}
For any $0\leq\alpha<1$, $0\leq t\leq 1$, if we have an estimate of:
\begin{equation}\label{preboundconditionforx2}
\mathbb{E}|x^j|^2\lesssim O\left(J^{-\alpha}\right)\,,
\end{equation}
for all $j$, then one can tighten it to: for any $0<\epsilon<\frac{1}{4}$, there is a constant $C_\epsilon$ independent of $J$ and $t$ such that
\begin{equation}\label{pestimationpx2}
\mathbb{E}|p^j|^2\leq C_\epsilon J^{-\frac{1}{2}-\frac{\alpha}{2}+\epsilon}\,,\ \mathbb{E}|x^j|^2\leq C_\epsilon J^{-\frac{1}{2}-\frac{\alpha}{2}+\epsilon}\,.
\end{equation}
for all $j$. Note for any $\alpha<1$, we can choose $\epsilon<1-\alpha$ to make $\frac{1}{2}+\frac{\alpha}{2}-\epsilon>\alpha$.
\end{lemma}
\begin{proof}
Firstly, by Lemma \ref{lem:tpoverlinex} equation \eqref{tpoverlinex1}, we have a rough estimate for $p^j,\overline{x}^j$
\begin{equation}\label{pestimationforpp2}
\mathbb{E}|p^j|^2\lesssim O\left(J^{-\alpha}\right),\ \mathbb{E}|\overline{x}|^2\lesssim O\left(J^{-\alpha}\right)\,.
\end{equation}
Similar to deriving~\eqref{eqn:firstevolution}, we subtract the two particle systems \eqref{eqn:SDE} and \eqref{eqn:mean_field}. With the calculation in Supp. B and Lemma \ref{lem:bound_tx}, for any $0<\epsilon<\frac{1}{4}$, there is a $J^*_\epsilon>0$ such that for $J>J^*_\epsilon$ and $0\leq t\leq 1$
\begin{equation}\label{eqn:firstactualevolution1star}
\begin{aligned}
&\frac{1}{J}\sum^J_{j=1}\frac{d\mathbb{E}|x^j|^2}{dt}\\
&\leq C_\epsilon J^{-\frac{1}{2}}\left(\left(\mathbb{E}\left|x^1\right|^2\right)^{\frac{2-\epsilon}{4}}+\left(\mathbb{E}\left|p^1\right|^2\right)^{\frac{2-\epsilon}{4}}\right)\\
&\ +C_\epsilon J^{-\frac{1}{4}-\frac{\alpha}{4}+\frac{\epsilon}{2}}\left(\left(\EE|\overline{x}|^{2}\right)^{\frac{2-\epsilon}{4}}+\left(\EE|p^1|^{2}\right)^{\frac{2-\epsilon}{4}}\right)\\
&\ +C_\epsilon J^{-\frac{1}{4}}\left(\left(\EE|x^1|^2\right)^{1-\epsilon}+\left(\EE|\overline{x}|^2\right)^{1-\epsilon}+\left(\EE|p^1|^2\right)^{1-\epsilon}\right)\\
&\ +C\left(\EE|x^1|^2+\EE|p^1|^2\right)+C_\epsilon J^{-\frac{1}{2}}\left(\mathbb{E}|x^1|^2\right)^{\frac{1}{2}}\\
&\ +C_\epsilon J^{-\frac{1}{2}-\frac{\alpha}{2}+\frac{\alpha\epsilon}{4}}\,,
\end{aligned}
\end{equation}
where $C_\epsilon$ is a constant independent of $J$ and $t$. Inserting \eqref{preboundconditionforx2},\eqref{pestimationforpp2} back into~\eqref{eqn:firstactualevolution1star}, we have the bounds for the first four terms:
\[
\begin{aligned}
&C_\epsilon J^{-\frac{1}{2}}\left(\left(\mathbb{E}\left|x^1\right|^2\right)^{\frac{2-\epsilon}{4}}+\left(\mathbb{E}\left|p^1\right|^2\right)^{\frac{2-\epsilon}{4}}\right)\\
\leq &C_\epsilon J^{-\frac{1}{2}-\frac{\alpha}{2}+\frac{\alpha\epsilon}{4}}\,\\
&C_\epsilon J^{-\frac{1}{4}-\frac{\alpha}{4}+\frac{\epsilon}{2}}\left(\left(\EE|\overline{x}|^{2}\right)^{\frac{2-\epsilon}{4}}+\left(\EE|p^1|^{2}\right)^{\frac{2-\epsilon}{4}}\right)\\
\leq &C_\epsilon J^{-\frac{1}{4}-\frac{\alpha}{4}+\frac{\epsilon}{2}}\left(\EE|x^1|^{2}\right)^{\frac{2-\epsilon}{4}}\,,\\
&C_\epsilon J^{-\frac{1}{4}}\left(\left(\EE|x^1|^2\right)^{1-\epsilon}+\left(\EE|\overline{x}|^2\right)^{1-\epsilon}+\left(\EE|p^1|^2\right)^{1-\epsilon}\right)\\
\leq &C_\epsilon J^{-\frac{1}{4}}\left(\EE|x^1|^2\right)^{1-\epsilon}\,\\
&C_\epsilon J^{-\frac{1}{2}}\left(\mathbb{E}|x^1|^2\right)^{\frac{1}{2}}\leq C_\epsilon J^{-\frac{1}{2}-\frac{\alpha}{2}}\,,
\end{aligned}
\] 
which implies, for $0<\epsilon<\frac{1}{4}$ and $J>J^\ast_\epsilon$:
\begin{equation*}
\begin{aligned}
\frac{d\mathbb{E}|x^1|^{2}}{dt}=&\frac{1}{J}\sum^J_{j=1}\frac{d\mathbb{E}|x^j|^2}{dt}\\
\leq&\ C_\epsilon J^{-\frac{1}{4}-\frac{\alpha}{4}+\frac{\epsilon}{2}}\left(\EE|x^1|^{2}\right)^{\frac{2-\epsilon}{4}}\\\
&+C_\epsilon J^{-\frac{1}{4}}\left(\EE|x^1|^2\right)^{1-\epsilon}+\EE|x^1|^{2}\\
&+J^{-\frac{1}{2}-\frac{\alpha}{2}+\frac{\alpha\epsilon}{4}}\,.
\end{aligned}
\end{equation*}

Similar to \eqref{Gronexample1}, define $\mathsf{X}^\beta=\mathbb{E}J^\beta|x^1|^{2}$, we have
\[
\begin{aligned}
\frac{d\mathsf{X}^\beta}{dt}\leq&\,C_\epsilon J^{-\frac{1}{4}-\frac{\alpha}{4}+\frac{\beta(2+\epsilon)}{4}}\left(\mathsf{X}^\beta\right)^{\frac{2-\epsilon}{4}}\\
&+C_\epsilon J^{-\frac{1}{4}+\epsilon\beta}\left(\mathsf{X}^\beta\right)^{1-\epsilon}\\
&+C_\epsilon\mathsf{X}^\beta+J^{-\frac{1}{2}-\frac{\alpha}{2}+\frac{\alpha\epsilon}{4}+\beta}\,,
\end{aligned}
\]
which implies 
\begin{equation}\label{xbetabound}
\begin{aligned}
\|\mathsf{X}^\beta\|_{L^\infty}\lesssim \max\{&O(1),J^{-\frac{1}{4}-\frac{\alpha}{4}+\frac{\beta(2+\epsilon)}{4}},\\
&J^{-\frac{1}{4}+\epsilon\beta},J^{-\frac{1}{2}-\frac{\alpha}{2}+\frac{\alpha\epsilon}{4}+\beta}\}
\end{aligned}\,.
\end{equation}
for $J>J^*_\epsilon$. Noting that 
\[
\|\mathsf{X}^\beta\|_{L^\infty}\leq (J^{\ast}_\epsilon)^\beta\sup_{0\leq t\leq 1}\EE|x^1_t|^2\leq (J^{\ast}_\epsilon)^\beta C\lesssim O(1)\,
\]
for all $J\leq J^\ast_\epsilon$ with constant $C$ stemming from the boundedness of Corollary \ref{cor:prebound}. We have \eqref{xbetabound} holds true for all $J>0$.
Therefore, we can choose $\beta=\frac{1+\alpha}{2+\epsilon}$ to obtain 
\[
\mathbb{E}|x^j|^2=\mathbb{E}|x^1|^{2}\lesssim O\left(J^{-\frac{1+\alpha}{2+\epsilon}}\right)
\]
for any $\epsilon<\frac{1}{4}$, which concludes \eqref{pestimationpx2}.\qed
\end{proof}

Finally, we are ready to prove Proposition~\ref{thm:vj_uj}.
\begin{proof}

We first note that by the definition of Wasserstein distance, for any $0\leq t\leq 1$
\[
\begin{aligned}
\mathbb{E}\left(W_2(M_{v_t},M_{u_t})\right)&\leq \left(\frac{1}{J}\sum^J_{j=1}\mathbb{E}|u^j_t-v^j_t|^2\right)^{\frac{1}{2}}\\
&=\left(\frac{1}{J}\sum^J_{j=1}\mathbb{E}|x^j_t|^2\right)^{\frac{1}{2}}
\end{aligned}\,,
\]
and thus the estimate~\eqref{eqn:W2_estimate} holds true once~\eqref{Expectationdifference} is shown. For that we directly apply Lemma~\ref{lem:bound_x}. Starting with $\alpha_0=0$ we recursively use the lemma, equation~\eqref{pestimationpx2} in particular, for
\[
\alpha_n=\frac{1}{2}+\alpha_{n-1}/2-\epsilon
\]
till the rate saturates to $\lim_{n\rightarrow\infty}\alpha_n=1-2\epsilon$. Since $\epsilon$ is an arbitrary small number, we conclude the proof.\qed
\end{proof}

\section{Acknowledgement}
The research of Q.L. and Z.D. was supported in part by National Science Foundation under award 1619778, 1750488 and Wisconsin Data Science Initiative. Both authors would like to thank Andrew Stuart for the helpful discussions.
\bibliographystyle{plain}
\bibliography{eki_2}
\appendix
\section{Moments bound of summation of indepedent mean-zero random variables}\label{sec:momentbound}
In this section, we prove a lemma which is used in proof of Lemma \ref{prop:highmomentforpac2}.
\begin{lemma}\label{lemma:momentboundindependent}
Assume $x_1,\cdots,x_J$ are $i.i.d$ random variables and satisfy (for $p\geq2$)
\[
\EE x_i=0,\quad \mathcal{L}_p=\EE |x_i|^p<\infty\,.
\]
Then we have
\[
\left(\EE\left|\sum^J_{j=1}x_j\right|^p\right)^{1/p}\leq CJ^{1/2}\,,
\]
where $C$ is a constant only depends on $\mathcal{L}_p$ and $p$.
\end{lemma}
\begin{proof} Without loss of generality, we assume $p$ is an even number and $J>p/2$. Then $\EE\left|\sum^J_{j=1}x_j\right|^p=\EE\left(\sum^J_{j=1}x_j\right)^{p}$.
\\
Since $\{x_i\}$ are independent with zero mean, we have
\[
\EE\left(\sum^J_{j=1}x_j\right)^{p}=\sum_{j_1+j_2+\cdots+j_J=p}\EE\left(x^{j_1}_1x^{j_2}_2\cdots x^{j_J}_J\right)\,,
\]
where $\{j_n\}$ should be non-negative integers and not equal to $1$ (otherwise $\EE x_i=0$ provides a trivial contribution).
\\
For each term in the summation, using generalization of H\"older's inequality, we have
\[
\EE\left(x^{j_1}_1x^{j_2}_2\cdots x^{j_J}_J\right)\leq \Pi^J_{n=1}(\EE|x_{n}|^p)^{j_n/p}=\mathcal{L}_p\,,
\]
which impies
\begin{equation}\label{Exjnumber}
\EE\left(\sum^J_{j=1}x_j\right)^{p}\leq \mathcal{L}_p\left(\sum_{j_1+j_2+\cdots+j_J=p}1\right)=\mathcal{L}_p|I_1|
\end{equation}
where
\[
I_1=\left\{\left(j_1,\cdots,j_J\right)\middle|j_n\in\mathbb{N}\setminus \{1\},\ \sum^J_{n=1}j_n=p\right\}
\]
and $|I_1|$ denotes the cardinality of the set $I_1$. 
\\
In $I_1$, if $j_n$ doesn't equal to zero, then $j_n$ is at least $2$, meaning there are at most $p/2$ non-trivial elements in the vector. Therefore, we have the following inequality
\begin{equation}\label{Inumber}
|I_1|\leq P(J,{p/2})|I_2|\leq J^{p/2}|I_{2}|\leq C(p)J^{p/2}\,.
\end{equation}
Here $P(J,{p/2})$ denotes the number of $p/2$-permutations in $J$ and is thus smaller than $J^{p/2}$, and $I_{2}$ is a new set defined by:
\[
I_{2} =  \left\{\left(i_1,\cdots,i_{p/2}\right)\middle|i_n\in\mathbb{N}^+\setminus \{1\},\ \sum^{p/2}_{n=1}i_n=p\right\}\,.
\]
Its cardinality does not have $J$ dependence and thus we bound it by $C(p)$, a constant depending on $p$ only.\qed
\end{proof}

\section{Bound of high moments of $\{u^j\}$}\label{appendixhighmoment}
\begin{proof}
For convenience, we omit the subscript $'t'$ in $u,\textbf{u},e,\textbf{e}$ etc. First, we prove the boundedness of $\mathbb{E}\left[\frac{1}{J}\sum^J_{j}|e^j|^2\right]^p$, which we will use later.
\begin{equation}\label{epbound}
\begin{aligned}
\mathbb{E}\left[\frac{1}{J}\sum^J_{j}|e^j|^2\right]^{p}&\leq \mathbb{E}\left[\sum^{K}_{m=1}\frac{1}{J}\sum^J_{j}|e^j_m|^2\right]^{p}\\
&\leq C_{p}\mathbb{E}\left(\sum^{K}_{m=1}\left[\frac{1}{J}\sum^J_{j}|e^j_m|^2\right]^{p}\right)\\
&\leq C_{p}V_{2p}(e)\leq C\,,
\end{aligned}
\end{equation}
which also implies
\begin{equation}\label{tepbound}
\mathbb{E}\left[\frac{1}{J}\sum^J_{j}|\textbf{e}^j|^2\right]^{p}\leq C\mathbb{E}\left[\frac{1}{J}\sum^J_{j}|e^j|^2\right]^{p}\leq C\,.
\end{equation}
Then, we first estimate $\EE|\tu^j|^{2p}$. Using Ito's formula, for fix $1\leq j\leq J$ and $p\geq1$, we obtain
\begin{equation}\label{tuestimate1}
\begin{aligned}
&d|\tu^j|^{2p}=-2p\left(|\tu^j|^{2(p-1)}\left\langle \tu^j,\Cov_{\tu}\tu^j\right\rangle\right)dt+\mathrm{\textbf{R}}\ dW^j_t\\
&+p\left(|\tu^j|^{2(p-1)}\left[\frac{1}{J^2}\sum^J_{i,k=1}\left\langle \te^i,\te^k\right\rangle^2\right]\right)dt\\
&+\frac{2p(p-1)}{J^2}\left(\left|\tu^j\right|^{2(p-2)}\sum^J_{i,k=1}\left\langle \tu^j,\te^i\right\rangle\left\langle \tu^j,\te^k\right\rangle\left\langle \te^i,\te^k\right\rangle\right)dt\\
&+2p\left(|\tu^j|^{2(p-1)}\left\langle \tu^j,\Cov_{\tu,\tr}\Gamma^{-\frac{1}{2}}\left(r-\Rm(u)\right)\right\rangle\right)dt\\
&+p\left(|\tu^j|^{2(p-1)}\left[\frac{1}{J^2}\sum^J_{i,k=1}\left\langle \te^i,\te^k\right\rangle\left\langle \tr^i,\tr^k\right\rangle\right]\right)dt\\
&+\frac{2p(p-1)}{J^2}\left(\left|\tu^j\right|^{2(p-2)}\sum^J_{i,k=1}\left\langle \tu^j,\te^i\right\rangle\left\langle \tu^j,\te^k\right\rangle\left\langle \tr^i,\tr^k\right\rangle\right)dt\,,\\
\end{aligned}
\end{equation}
where $\mathrm{\textbf{R}}$ is the coefficient before Brownian motion. The first term is negative. To complete the computation, we need to provide the bound for the rest. The second term is bounded by:
\begin{equation*}\label{Boundtu2}
\begin{aligned}
&\mathbb{E}\left(|\tu^j|^{2(p-1)}\left[\frac{1}{J^2}\sum^J_{i,k=1}\left\langle \te^i,\te^k\right\rangle^2\right]\right)\\
\leq&\mathbb{E}\left(|\tu^j|^{2(p-1)}\left[\frac{1}{J}\sum^J_{i=1}|\te^i|^2\right]^2\right)\\
\leq&\left(\mathbb{E}|\tu^j|^{2p}\right)^{(p-1)/p}\left(\EE\left[\frac{1}{J}\sum^J_{i=1}|\te^i|^2\right]^{2p}\right)^{1/p}\,.
\end{aligned}
\end{equation*}
The third term is bounded by:
\begin{equation*}
\begin{aligned}
&\frac{1}{J^2}\EE\left(\left|\tu^j\right|^{2(p-2)}\sum^J_{i,k=1}\left\langle \tu^j,\te^i\right\rangle\left\langle \tu^j,\te^k\right\rangle\left\langle \te^i,\te^k\right\rangle\right)\\
\leq&\mathbb{E}\left(|\tu^j|^{2(p-1)}\left[\frac{1}{J}\sum^J_{i=1}|\te^i|^2\right]^2\right)\\
\leq&\left(\mathbb{E}|\tu^j|^{2p}\right)^{(p-1)/p}\left(\EE\left[\frac{1}{J}\sum^J_{i=1}|\te^i|^2\right]^{2p}\right)^{1/p}\,.
\end{aligned}
\end{equation*}
And similarly, the rests are bounded by:
\begin{equation*}
\begin{aligned}
&\mathbb{E}\left(|\tu^j|^{2(p-1)}\left\langle \tu^j,\Cov_{\tu,\tr}\Gamma^{-\frac{1}{2}}\left(r-\Rm(u)\right)\right\rangle\right)\\
\leq &C\mathbb{E}\left(|\tu^j|^{2(p-\frac{1}{2})}\left[\frac{1}{J}\sum^J_{k=1}|\te^k|^2\right]^{\frac{1}{2}}\right)\\
\leq &C\left(\mathbb{E}|\tu^j|^{2p}\right)^{(p-\frac{1}{2})/p}\left(\EE\left[\frac{1}{J}\sum^J_{i=1}|\te^i|^2\right]^{p}\right)^{1/(2p)}
\end{aligned}
\end{equation*}
and
\begin{equation*}
\begin{aligned}
&\mathbb{E}\left(|\tu^j|^{2(p-1)}\left[\frac{1}{J^2}\sum^J_{i,k=1}\left\langle \te^i,\te^k\right\rangle\left\langle \tr^i,\tr^k\right\rangle\right]\right)\\
\leq &C\mathbb{E}\left(|\tu^j|^{2(p-1)}\left[\frac{1}{J}\sum^J_{i=1}|\te^i|^2\right]\right)\\
\leq &C\left(\mathbb{E}|\tu^j|^{2p}\right)^{(p-1)/p}\left(\EE\left[\frac{1}{J}\sum^J_{i=1}|\te^i|^2\right]^{p}\right)^{1/p}
\end{aligned}
\end{equation*}
and
\begin{equation*}\label{Boundtu6}
\begin{aligned}
&\frac{1}{J^2}\EE\left(\left|\tu^j\right|^{2(p-2)}\sum^J_{i,k=1}\left\langle \tu^j,\te^i\right\rangle\left\langle \tu^j,\te^k\right\rangle\left\langle \tr^i,\tr^k\right\rangle\right)\\
\leq &C\mathbb{E}\left(|\tu^j|^{2(p-1)}\left[\frac{1}{J}\sum^J_{i=1}|\te^i|^2\right]\right)\\
\leq &C\left(\mathbb{E}|\tu^j|^{2p}\right)^{(p-1)/p}\left(\EE\left[\frac{1}{J}\sum^J_{i=1}|\te^i|^2\right]^{p}\right)^{1/p}\,.
\end{aligned}
\end{equation*}
Plug all these inequalities back in \eqref{tuestimate1}, and utilize \eqref{tepbound}, we have:
\begin{equation}\label{estimationfortu}
\frac{d\mathbb{E}|\tu^j|^{2p}}{dt}\leq 2C\left(\mathbb{E}|\tu^j|^{2p}\right)^{(p-1)/p}\Rightarrow \mathbb{E}|\tu^j|^{2p}\leq C\,.
\end{equation}

Then, to deal with $\EE|u^j|^{2p}$, we use Ito's formula similarly, for fix $1\leq j\leq J$ and $p\geq1$, we obtain
\[
\begin{aligned}
&\frac{d|u^j|^{2p}}{dt}=-2p\left(|u^j|^{2(p-1)}\left\langle u^j,\Cov_{u,\tu}\tu^j\right\rangle\right)dt+\mathrm{R} dW^j_t\\
&+p\left(|u^j|^{2(p-1)}\left[\frac{1}{J^2}\sum^J_{i,k=1}\left\langle e^i,e^k\right\rangle\left\langle \textbf{e}^k,\textbf{e}^i\right\rangle\right]\right)dt\\
&+\frac{2p(p-1)}{J^2}\left(\left|u^j\right|^{2(p-2)}\sum^J_{i,k=1}\left\langle u^j,e^i\right\rangle\left\langle u^j,e^k\right\rangle\left\langle \te^i,\te^k\right\rangle\right)dt\\
&+2p\left(|u^j|^{2(p-1)}\left\langle u^j,\Cov_{u,\tr}\Gamma^{-\frac{1}{2}}\left(r-\Rm(u)\right)\right\rangle\right)dt\\
&+p\left(|u^j|^{2(p-1)}\left[\frac{1}{J^2}\sum^J_{i,k=1}\left\langle e^i,e^k\right\rangle\left\langle \tr^i,\tr^k\right\rangle\right]\right)dt\\
&+\frac{2p(p-1)}{J^2}\left(\left|u^j\right|^{2(p-2)}\sum^J_{i,k=1}\left\langle u^j,e^i\right\rangle\left\langle u^j,e^k\right\rangle\left\langle \tr^i,\tr^k\right\rangle\right)dt\,,
\end{aligned}
\]
where $\mathrm{R}$ is the coefficient before Brownian motion. The six terms are considered separately:
\begin{enumerate}[itemindent=2em,topsep=0pt,itemsep=-1ex,partopsep=1ex,parsep=1ex,label=\arabic*]
\item[Term 1]
\[
\begin{aligned}
&\left|\mathbb{E}\left(|u^j|^{2(p-1)}\left\langle u^j,\Cov_{u,\tu}\tu^j\right\rangle\right)\right|\\
\leq &\mathbb{E}\left(|u^j|^{2p-\frac{1}{2}}\frac{1}{J}\sum^J_{k=1}|e^k||\te^k||\tu^j|\right)\\
\leq &\left(\mathbb{E}|u^j|^{2p}\right)^{(2p-\frac{1}{2})/(2p)}\left(\EE\left(\frac{1}{J}\sum^J_{k=1}|e^k||\te^k||\tu^j|\right)^{4p}\right)^{1/(4p)}\\
\leq &C\left(\mathbb{E}|u^j|^{2p}\right)^{(2p-\frac{1}{2})/(2p)}\,,
\end{aligned}
\]
where in the last inequality we use \eqref{epbound},\eqref{tepbound} and \eqref{estimationfortu} with H\"older's inequality.
\item[Term 2]
\[
\begin{aligned}
&\left|\mathbb{E}\left(|u^j|^{2(p-1)}\left[\frac{1}{J^2}\sum^J_{i,k=1}\left\langle e^i,e^k\right\rangle\left\langle \textbf{e}^k,\textbf{e}^i\right\rangle\right]\right)\right|\\
\leq&\ C\mathbb{E}\left(|u^j|^{2(p-1)}\left[\frac{1}{J^2}\sum^J_{i,k=1}|e^i|
|e^k||\textbf{e}^i||\textbf{e}^k|\right]\right)\\
\leq&\ C\mathbb{E}\left(|u^j|^{2(p-1)}\left(\frac{1}{J}\sum^J_{i=1}|e^i|^2\right)
\left(\frac{1}{J}\sum^J_{k=1}|e^k|^2\right)\right)\\
\leq&\ C\mathbb{E}\left(|u^j|^{2p}\right)^{(p-1)/p} \left(\mathbb{E}\left[\frac{1}{J}\sum^J_{i=1}|e^i|^2\right]^{2p}\right)^{1/p}\\
\leq&\ C\mathbb{E}\left(|u^j|^{2p}\right)^{(p-1)/p} \left(\mathbb{E}\left[\sum^K_{m=1}\frac{1}{J}\sum^J_{i=1}|e^i_m|^2\right]^{2p}\right)^{1/p}\\
\leq&\ C\mathbb{E}\left(|u^j|^{2p}\right)^{(p-1)/p} \left(\mathbb{E}\sum^K_{m=1}\left[\frac{1}{J}\sum^J_{i=1}|e^i_m|^2\right]^{2p}\right)^{1/p}\\
\leq&\  CV_{4p}^{1/p}(e_0)\mathbb{E}\left(|u^j|^{2p}\right)^{(p-1)/p}\,.
\end{aligned}
\]
\item[Term 3]
\[
\begin{aligned}
&\left|\EE\left(\left|u^j\right|^{2(p-2)}\left[\frac{1}{J^2}\sum^J_{i,k=1}\left\langle u^j,e^i\right\rangle\left\langle u^j,e^k\right\rangle\left\langle \te^i,\te^k\right\rangle\right]\right)\right|\\
\leq&\,C\mathbb{E}\left(|u^j|^{2(p-1)}\left[\frac{1}{J^2}\sum^J_{i,k=1}|e^i|
|e^k||\textbf{e}^i||\textbf{e}^k|\right]\right)\\
\leq&\  CV_{4p}^{1/p}(e_0)\mathbb{E}\left(|u^j|^{2p}\right)^{(p-1)/p}\,.
\end{aligned}
\]
\item[Term 4]
\[
\begin{aligned}
&\left|\mathbb{E}\left(|u^j|^{2(p-1)}\left\langle u^j,\Cov_{u,\tr}\Gamma^{-\frac{1}{2}}\left(r-\Rm(u)\right)\right\rangle\right)\right|\\
\leq &M^2\mathbb{E}\left(|u^j|^{2p-\frac{1}{2}}\frac{1}{J}\sum^J_{k=1}|e^k|\right)\\
\leq &\left(\mathbb{E}|u^j|^{2p}\right)^{(2p-\frac{1}{2})/(2p)}\left(\EE\left(\frac{1}{J}\sum^J_{k=1}|e^k|\right)^{4p}\right)^{1/(4p)}\\
\leq &C\left(\mathbb{E}|u^j|^{2p}\right)^{(2p-\frac{1}{2})/(2p)}\,,
\end{aligned}
\]
where in the last inequality we use \eqref{epbound} and \eqref{estimationfortu} with H\"older's inequality.
\item[Term 5]
\[
\begin{aligned}
&\left|\mathbb{E}\left(|u^j|^{2(p-1)}\left[\frac{1}{J^2}\sum^J_{i,k=1}\left\langle e^i,e^k\right\rangle\left\langle \textbf{r}^k,\textbf{r}^i\right\rangle\right]\right)\right|\\
\leq&\ CM^2\mathbb{E}\left(|u^j|^{2(p-1)}\left(\frac{1}{J}\sum^J_{i=1}|e^i|^2\right)\right)\\
\leq&\ C\mathbb{E}\left(|u^j|^{2p}\right)^{(p-1)/p} \left(\mathbb{E}\left[\frac{1}{J}\sum^J_{i=1}|e^i|^2\right]^{p}\right)^{1/p}\\
\leq&\ C\mathbb{E}\left(|u^j|^{2p}\right)^{(p-1)/p} \left(\mathbb{E}\left[\sum^K_{m=1}\frac{1}{J}\sum^J_{i=1}|e^i_m|^2\right]^{p}\right)^{1/p}\\
\leq&\ C\mathbb{E}\left(|u^j|^{2p}\right)^{(p-1)/p} \left(\mathbb{E}\sum^K_{m=1}\left[\frac{1}{J}\sum^J_{i=1}|e^i_m|^2\right]^{p}\right)^{1/p}\\
\leq&\  CV_{2p}^{1/p}(e_0)\mathbb{E}\left(|u^j|^{2p}\right)^{(p-1)/p}\,.
\end{aligned}
\]
\item[Term 6]
\[
\begin{aligned}
&\left|\EE\left(\left|u^j\right|^{2(p-2)}\left[\frac{1}{J^2}\sum^J_{i,k=1}\left\langle u^j,e^i\right\rangle\left\langle u^j,e^k\right\rangle\left\langle \tr^i,\tr^k\right\rangle\right]\right)\right|\\
\leq&\,C\mathbb{E}\left(|u^j|^{2(p-1)}\left[\frac{1}{J^2}\sum^J_{i,k=1}|e^i|
|e^k||\textbf{r}^i||\textbf{r}^k|\right]\right)\\
\leq&\,CV_{2p}^{1/p}(e_0)\mathbb{E}\left(|u^j|^{2p}\right)^{(p-1)/p}\,.
\end{aligned}
\]
\end{enumerate}
By Lemma \ref{anotherestimate}, we obtain the boundedness for $\mathbb{E}\left\|u^j\right\|^{2p}_2$ .Then to prove the second inequality of \eqref{highmomentu}, it suffices to prove
\[
\left(\mathbb{E}\left\|\Cov_{u_t}\right\|^p_2\right)^{1/p}\leq C_p\,,
\]
which is a direct result by expansion of $\Cov_{u_t}$ and triangle inequality:
\[
\begin{aligned}
\left(\mathbb{E}\left\|\Cov_{u_t}\right\|^p_2\right)^{1/p}&\leq \frac{1}{J}\sum^J_{j=1}\left(\mathbb{E}\left\|(u^j-\overline{u})\otimes(u^j-\overline{u})\right\|^p_2\right)^{1/p}\\&
\leq \frac{1}{J}\sum^J_{j=1}\left(\mathbb{E}\left|u^j-\overline{u}\right|^{2p}\right)^{1/p}\leq C\,.
\end{aligned}
\]
Here the last inequality comes from each term of the sum has a bound
\[
\begin{aligned}
&\left(\mathbb{E}\left|u^j-\overline{u}\right|^{2p}\right)^{1/p}\leq \left[\left(\mathbb{E}\left|u^j-\overline{u}\right|^{2p}\right)^{\frac{1}{2}p}\right]^{2}\\
\leq &\left[\frac{J-1}{J}\mathbb{E}\left(|u^j|^{2p}\right)^{\frac{1}{2}p}+\frac{1}{J}\sum^J_{k\neq j}\mathbb{E}\left(|u^k|^{2p}\right)^{\frac{1}{2}p}\right]^{2}\leq C\,.
\end{aligned}
\]
\qed
\end{proof}
\end{document}